\documentclass[11pt,letterpaper,reqno]{amsart}
\usepackage[margin=1in]{geometry}
\usepackage{color}
\usepackage[latin1]{inputenc}
\usepackage{amsmath,amsfonts,amsthm,epsfig,graphicx,amssymb}
\usepackage{cleveref}
\usepackage{esint}
\usepackage{caption}
\usepackage{enumitem}
\usepackage{overpic}

\renewcommand{\S}{\Sigma}

\newcommand{\dist}{{\rm dist}}
\newcommand{\loc}{{\rm loc}}

\newcommand{\spt}{{\rm spt}}

\newcommand{\weakstar}{\stackrel{*}{\rightharpoonup}}

\newcommand{\cl}{\mathrm{cl}\,}

\newcommand{\dive}{\mathrm{div}\,}

\newcommand\restr[2]{{
  \left.\kern-\nulldelimiterspace 
  #1 
  \right|_{#2} 
  }}

\newcommand{\mres}{\mathbin{\vrule height 1.6ex depth 0pt width 
0.13ex\vrule height 0.13ex depth 0pt width 1.3ex}}

\theoremstyle{plain}
\newtheorem{theorem}{Theorem}[section]
\newtheorem{lemma}[theorem]{Lemma}
\newtheorem{corollary}[theorem]{Corollary}

\newtheorem*{theorem*}{Theorem}
\newtheorem*{corollary*}{Corollary}

\theoremstyle{definition}
\newtheorem{definition}[theorem]{Definition}

\newtheorem{remark}[theorem]{Remark}

\newtheorem*{notation*}{Notation}

\numberwithin{equation}{section}
\numberwithin{figure}{section}

\newcommand{\mm}{\mathbf{m}}
\renewcommand{\S}{\mathcal{S}}

\allowdisplaybreaks

\title[]{Regularity for Minimizers of a Planar Partitioning Problem with Cusps}

\author[M. Novack]{Michael Novack}
\address{Department of Mathematics, Louisiana State University, 303 Lockett Hall, Baton Rouge, LA 70803, United States of America}
\email{mnovack@lsu.edu}

\begin{document}

\begin{abstract}
We study the regularity of minimizers for a variant of the soap bubble cluster problem:
\begin{align*}
 \min \sum_{\ell=0}^N c_{\ell} P( S_\ell)\,, 
\end{align*}
where $c_\ell>0$, among partitions $\{S_0,\dots,S_N,G\}$ of $\mathbb{R}^2$ satisfying $|G|\leq \delta$ and an area constraint on each $S_\ell$ for $1\leq \ell \leq N$. If $\delta>0$, we prove that for any minimizer, each $\partial S_{\ell}$ is $C^{1,1}$ and consists of finitely many curves of constant curvature. Any such curve contained in $\partial S_{\ell} \cap \partial S_{m}$ or $\partial S_\ell \cap \partial G$ can only terminate at a point in $\partial G \cap \partial S_\ell \cap \partial S_{m}$ at which $G$ has a cusp. We also analyze a similar problem on the unit ball $B$ with a trace constraint instead of an area constraint and obtain analogous regularity up to $\partial B$. Finally, in the case of equal coefficients $c_\ell$, we completely characterize minimizers on the ball for small $\delta$: they are perturbations of minimizers for $\delta=0$ in which the triple junction singularities, including those possibly on $\partial B$, are ``wetted" by $G$. 
\end{abstract}

\maketitle

\setcounter{tocdepth}{2}

\section{Introduction}
\subsection{Overview} A classical problem in the calculus of variations is the soap bubble cluster problem, which entails finding the configuration, or cluster, of least area separating $N$ regions with prescribed volumes, known as chambers. Various generalizations have been studied extensively as well and may involve different coefficients penalizing the interfaces between pairs of regions (the immiscible fluids problem) or anisotropic energies. The existence of minimal clusters and almost everywhere regularity for a wide class of problems of this type were obtained by Almgren in the foundational work \cite{Alm76}. The types of singularities present in minimizers in the physical dimensions are described by Plateau's laws, which were verified in $\mathbb{R}^3$ by Taylor \cite{Taylor}. In the plane, regions in a minimizing cluster are bounded by finitely many arcs of constant curvature meeting at $120^{\circ}$ angles \cite{Morgan94}. We refer to the book \cite{MorganBook} for further discussion on the literature for soap bubble clusters.
\par
In this article we study the interaction of the regularity/singularities of 2D soap bubbles with other physical properties such as thickness. Soap bubbles are generally modeled as surfaces, or ``dry" soap bubbles. This framework is quite natural for certain questions, e.g. singularity analysis as observed above, but it does not capture features related to thickness or volume of the soap. 
Issues such as which other types of singularities can be stabilized by ``wetting" the film \cite{Hutzler,WeairePhelan} require the addition of a small volume parameter to the model corresponding to the enclosed liquid; see for example \cite{BrakkeMorgan,Brakke05}. In the context of least-area surfaces with fixed boundary (Plateau problem), the authors in \cite{MSS,KinMagStu22,KMS2,KMS3} have formulated a soap film capillarity model that selects surface tension energy minimizers enclosing a small volume and spanning a given wire frame. The analysis of minimizers is challenging, for example due to the higher multiplicity surfaces that arise if the thin film ``collapses." 
\par
Here we approach these issues through the regularity analysis of minimizers of a version of the planar minimal cluster problem. In the model, there are $N$ chambers of fixed area (the soap bubbles) and an exterior chamber whose perimeters are penalized, and there is also an un-penalized region $G$ of small area at most $\delta>0$. This region may be thought of as the ``wet" part of the soap film where soap accumulates (see Remarks \ref{interpretation remark}-\ref{remark on soft constraint} and \ref{wetting remark}). Our first main result, Theorem \ref{main regularity theorem delta positive all of space}, is a sharp regularity result for minimizers: each of the $N$ chambers as well as the exterior chamber have $C^{1,1}$ boundary, while $\partial G$ is regular away from finitely many cusps. In particular, each bubble is regular despite the fact that the bubbles in the $\delta\to 0$ limit may exhibit singularities. We also study a related problem on the ball in which the area constraints on the chambers are replaced by boundary conditions on the circle and prove a similar theorem up to the boundary (Theorem \ref{main regularity theorem delta positive}). As a consequence, in Theorem \ref{resolution for small delta corollary}, we completely resolve minimizers on the ball for small $\delta$ in terms of minimizers for the limiting ``dry" problem: near each triple junction singularity of the limiting minimizer, there is a component of $G$ ``wetting" the 
 singularity and bounded by three circular arcs meeting in cusps inside the ball and corners or cusps at the boundary; see Figure \ref{perturbation figure}.

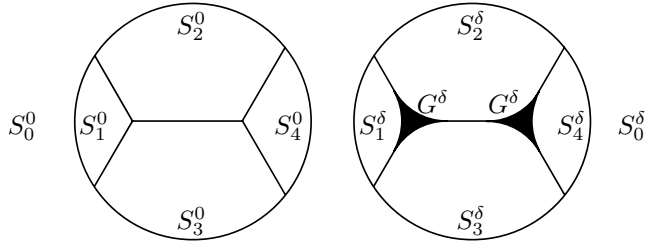
\begin{figure}
\begin{overpic}[scale=0.4,unit=1mm]{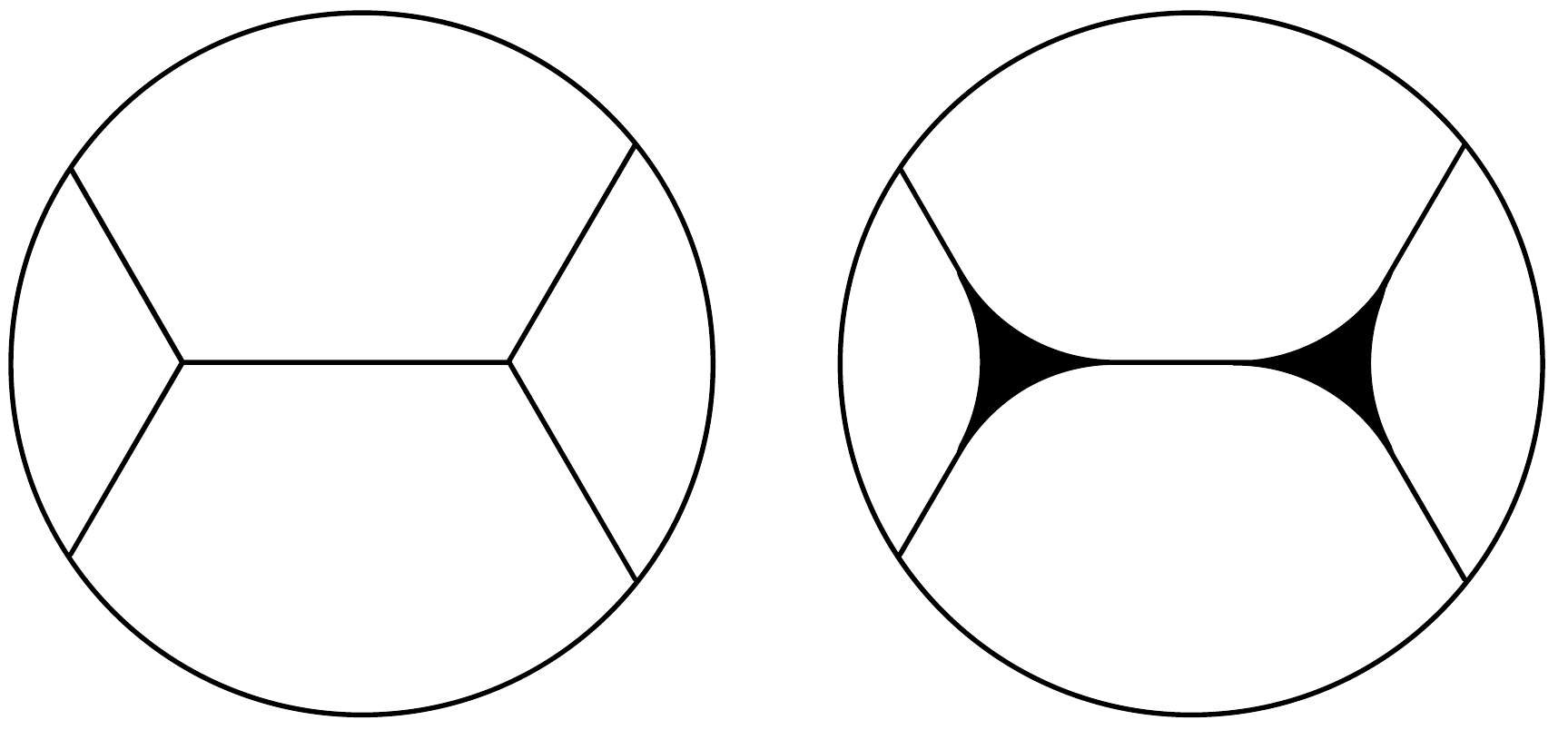}
    \put(-12,21){\small{$S_0^0$}}
    \put(1.5,21){\small{$S_1^0$}}
    \put(54.5,21){\small{$S_1^\delta$}}
    \put(20,40){\small{$S_2^0$}}
    \put(73,40){\small{$S_2^\delta$}}
    \put(20,3){\small{$S_3^0$}}
    \put(73,3){\small{$S_3^\delta$}}
    \put(38.5,21){\small{$S_4^0$}}
    \put(92,21){\small{$S_4^\delta$}}
    \put(103.5,21){\small{$S_0^\delta$}}
    \put(65.5,25){\small{$G^\delta$}}
    \put(79,25){\small{$G^\delta$}}
\end{overpic}
\caption{On the left is a minimizing cluster $\S^0$ for the $\delta=0$ problem on the ball with chambers $S_\ell^0$. On the right is a minimizer $\S^\delta$ for small $\delta$, with $|G^\delta|=\delta$. Near the triple junctions of $\S^0$, $\partial G^\delta$ consists of three circular arcs meeting in cusps; see Theorem \ref{resolution for small delta corollary}. }\label{perturbation figure}
\end{figure}

\subsection{Statement of the problem} For an $(N+2)$-tuple $\mathcal{S}=(S_0,S_1,\dots, S_N, G)$ of disjoint sets of finite perimeter partitioning $\mathbb{R}^2$ ($N\geq 2$), called a cluster, we study minimizers of the energy
\begin{align}\notag
    \mathcal{F}(\S):= \sum_{\ell=0}^N c_{\ell} P( S_\ell)\,,\qquad c_\ell>0\quad \forall 0 \leq \ell \leq N\,,
\end{align}
among two admissible classes. First, we consider the problem on all of space
\begin{align}\label{all of space problem}
    \inf_{\S \in \mathcal{A}_\delta^{\mm}} \mathcal{F}(\S)\,,
\end{align}
where the admissible class $\mathcal{A}_\delta^{\mm}$ consists of all clusters satisfying 
\begin{align}\label{lebesgue constraint}
    |G|=|\mathbb{R}^2 \setminus \cup_{\ell=0}^N S_\ell| \leq \delta
\end{align}
and, for some fixed $\mm\in (0,\infty)^N$, $(|S_1|,\dots,|S_N|)=\mm$. We also consider a related problem on the unit ball $B=\{(x,y):x^2+y^2<1\}$. We study the minimizers of
\begin{align}\label{ball problem}
    \inf_{\S \in \mathcal{A}_\delta^h} \mathcal{F}(\S)\,,
\end{align}
where $\mathcal{A}_\delta^h$ consists of all clusters such that, for fixed $h\in BV(\partial B;\{1,\dots,N\})$,
\begin{align}\label{trace constraint}
    S_\ell \cap \partial B = \{x\in \partial B : h(x) = \ell \}\textup{ for $1\leq \ell \leq N$ in the sense of traces}\,,
\end{align}
$S_0=\mathbb{R}^2\setminus B$ is the exterior chamber, and $G$ satisfies \eqref{lebesgue constraint}. We remark that since $\mathcal{A}_\delta^{\mm}\subset \mathcal{A}_{\delta'}^{\mm}$ and $\mathcal{A}_\delta^h \subset \mathcal{A}_{\delta'}^h$ if $\delta<\delta'$, the minimum energy decreases in $\delta$ for both \eqref{all of space problem} and \eqref{ball problem}.
\par
The main energetic mechanism at work in \eqref{all of space problem} that distinguishes it from the classical minimal cluster problem is that the set $G$ prohibits the creation of corners in the chambers $S_\ell$. If $r\ll 1$, the amount of perimeter saved by smoothing out a corner of $S_\ell$ in $B_r(x)$ using the set $G$ scales like $r$, and this can be accomplished while simultaneously preserving the area constraint by fixing areas elsewhere with cost $\approx r^2$ \cite[VI.10-12]{Alm76}. On the other hand, the regularizing effect of $G$ only extends to the other chambers and not to its own boundary since its perimeter is not penalized.



\subsection{Main results} We obtain optimal regularity results for minimizers of \eqref{all of space problem} and \eqref{ball problem}. In addition, for the problem with equal weights $c_\ell$, we completely resolve minimizers of \eqref{ball problem} for small $\delta>0$ in terms of minimizers for $\delta=0$. In the following theorems and throughout the paper, the term ``arc of constant curvature" may refer to either a single circle arc or a straight line segment.

\begin{theorem}[Regularity on $\mathbb{R}^2$ for $\delta>0$]\label{main regularity theorem delta positive all of space}
    If $\S^\delta$ is a minimizer for $\mathcal{F}$ among $\mathcal{A}_\delta^{\mm}$ for $\delta>0$, then $\partial S_\ell^\delta$ is $C^{1,1}$ for each $\ell$, and  there exists $\kappa^\delta_{\ell m}$ such that each $\partial S_\ell^\delta \cap \partial S_m^\delta$ is a finite union of arcs of constant curvature $\kappa_{\ell m}^\delta$ that can only terminate at a point in $\partial S_\ell^\delta \cap \partial S_m^\delta \cap \partial G^\delta$. Referring to those points in $\partial S_\ell^\delta \cap \partial S_m^\delta \cap \partial G^\delta$ as cusp points, there exist $\kappa_\ell^\delta$ for $0\leq \ell \leq N$ such that
$\partial S_\ell^\delta \cap \partial G^\delta$ is a finite union of arcs of constant curvature $\kappa_\ell^\delta$, each of which can only terminate at a cusp point where $\partial S_\ell^\delta\cap \partial G^\delta $ and $\partial S_m^\delta\cap \partial G^\delta $ meet a component of $\partial S_\ell^\delta \cap \partial S_m^\delta$ tangentially.
\end{theorem}


\begin{remark}[Interpretation of $G^\delta$]\label{interpretation remark}
For the case $c_\ell =1$, a possible reformulation of \eqref{all of space problem} that views the interfaces as thin regions of liquid rather than surfaces is
\begin{align}\label{thin film cluster formulation}
    \inf\{ \mathcal{F}(\S) : \S\in \mathcal{A}_\delta^{\mm},\, \mbox{$S_\ell$ open $\forall \ell$, $\cl S_\ell \cap \cl S_m=\emptyset$ $\forall \ell \neq m$} \}\,.
\end{align}
This is because if $\S$ belongs to this class, then each bubble $S_\ell$ for $1\leq \ell \leq N$ must be separated from the others and the exterior chamber $S_0$ by the soap $G$, and $\mathcal{F}(\S) = P(G)$, which is the energy of the soap coming from surface tension. Theorem \ref{main regularity theorem delta positive all of space} allows for a straightforward construction showing that in fact, \eqref{all of space problem} and \eqref{thin film cluster formulation} are equivalent, in that a minimizer for \eqref{all of space problem} can be approximated in energy by clusters in the smaller class \eqref{thin film cluster formulation}. Therefore, for a minimizer $\S^\delta$ of \eqref{all of space problem}, $G^\delta$ can be understood as the ``wet" part of the interfaces between bubbles where soap accumulates in the limit of a minimizing sequence for \eqref{thin film cluster formulation}, as opposed to $\partial S_\ell^\delta \cap \partial S_m^\delta$ which is the ``dry" part; see Figure \ref{double bubble}. 
\end{remark}

\begin{remark}[Constraint on $G^\delta$]\label{remark on soft constraint}
    We have incorporated $G^\delta$ with a soft constraint $|G^\delta|\leq \delta$ rather than a hard constraint $|G^\delta|=\delta$ to allow the minimizers to ``select" the area of $G^\delta$. A consequence of Theorem \ref{main regularity theorem delta positive all of space} is that if some minimizer $\S^0$ of \eqref{all of space problem} for $\delta=0$ has a singularity, then every minimizer $\S^\delta$ for given $\delta>0$ satisfies $|G^\delta|>0$. Indeed, if $|G^\delta|=0$, then $\mathcal{F}(\S^0)\leq \mathcal{F}(\S^\delta)=\inf_{\mathcal{A}_\delta^{\mm}}\mathcal{F}$, so that $\S^0$ is minimal among $\mathcal{A}_\delta^{\mm}$ and the regularity in Theorem \ref{main regularity theorem delta positive all of space} for $\S^0$ yields a contradiction. As we prove in Theorem \ref{resolution for small delta corollary}, the minimizer on the ball for small $\delta$ and equal coefficients saturates the inequality $|G^\delta|\leq \delta$, and we suspect this should hold in generality for \eqref{all of space problem} and \eqref{ball problem} with small $\delta$.
\end{remark}

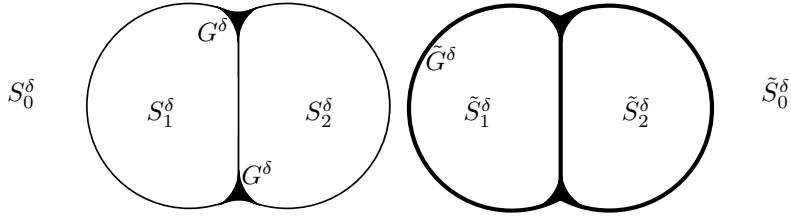
\begin{figure}
\begin{overpic}[scale=0.4,unit=1mm]{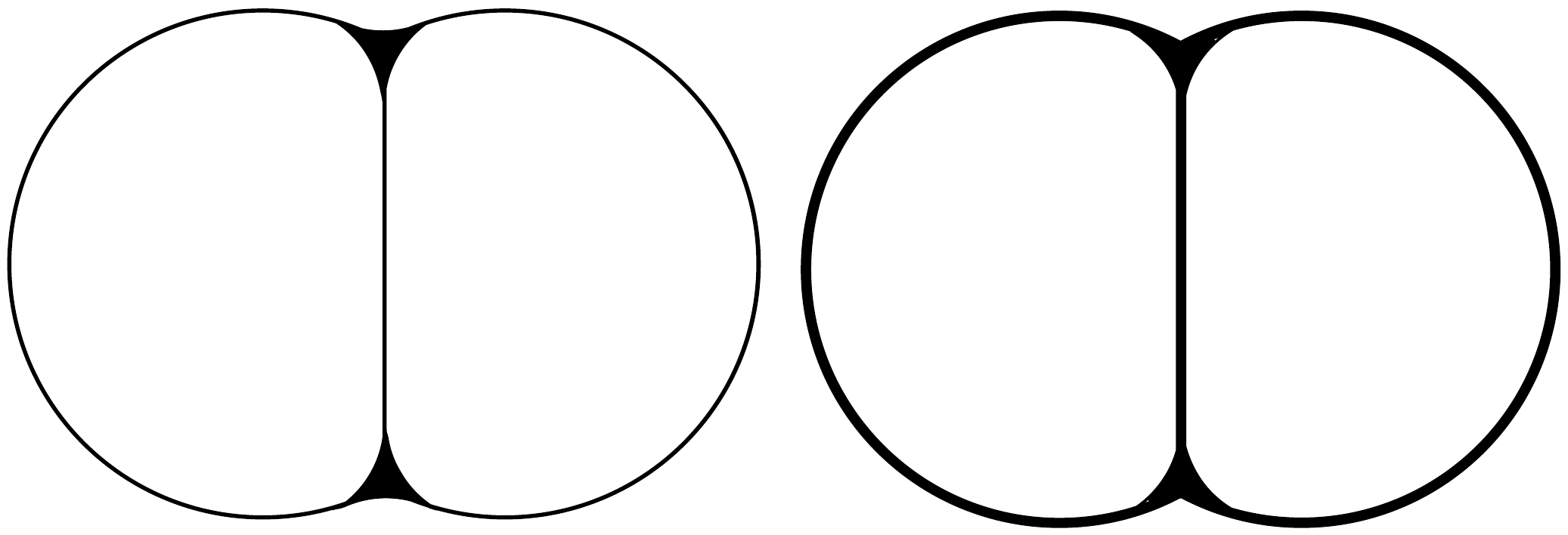}
    \put(10,15){\small{$S_1^\delta$}}
    \put(35,15){\small{$S_2^\delta$}}
    \put(60,15){\small{$\tilde{S}_1^\delta$}}
    \put(85,15){\small{$\tilde{S}_2^\delta$}}
    \put(-12,18){\small{$S_0^\delta$}}
    \put(107,18){\small{$\tilde{S}_0^\delta$}}
    \put(18,27.5){\small{$G^\delta$}}
    \put(25,5){\small{$G^\delta$}}
    \put(54,23){\small{$\tilde{G}^\delta$}}
\end{overpic}
\caption{On the left is a double 
 bubble-type configuration $\S^\delta$ with singularities wetted by $G^\delta$. $\S^\delta$ can be approximated by $\tilde{\S}^\delta$, where $\tilde{G}^\delta$ wets the entire interface.}\label{double bubble}
\end{figure}


We turn now to our results regarding the problem \eqref{ball problem} on the ball. Here regularity holds up to the boundary $\partial B$, at which $G^\delta$ may have corners, rather than cusps, at jump points of $h$. 

\begin{theorem}[Regularity on the Ball for $\delta>0$]\label{main regularity theorem delta positive}
    If $\S^\delta$ is a minimizer for $\mathcal{F}$ among $\mathcal{A}_\delta^h$ for $\delta>0$, then for $\ell,m>0$, $\partial S_\ell^\delta$ is $C^{1,1}$ except at jump points of $h$, and $\partial S_\ell^\delta \cap \partial S_m^\delta \cap B$ is a finite union of line segments terminating on $\partial B$ at a jump point of $h$ between $\ell$ and $m$ or at a point in $\partial S_\ell^\delta \cap \partial S_m^\delta \cap \partial G^\delta \cap  B$. Referring to those points in $\partial S_\ell^\delta \cap \partial S_m^\delta \cap \partial G^\delta \cap  B$ and $\partial S_\ell^\delta \cap \partial S_m^\delta \cap \partial G^\delta \cap  \partial B$ as cusp and corner points, respectively, there exist $\kappa_\ell^\delta$ for $1\leq \ell\leq N$ such that
\begin{align}\label{curvature condition}
    c_1 \kappa_1^\delta=c_2 \kappa_2^\delta=\cdots =c_N \kappa_N^\delta 
\end{align}
and $\partial S_\ell^\delta \cap \partial G^\delta$ consists of a finite union of arcs of constant curvature $\kappa_\ell^\delta$, each of whose two endpoints are either a cusp point in $B$ or a corner point in $\partial B$ at a jump point of $h$. Furthermore, at cusp points, $\partial S_\ell^\delta\cap \partial G^\delta $ and $\partial S_m^\delta\cap \partial G^\delta $ meet a segment of $\partial S_\ell^\delta \cap \partial S_m^\delta$ tangentially. Finally, any connected component of $S_\ell^\delta$ for $1\leq \ell \leq N$ is convex.
\end{theorem}


\begin{remark}
    In the case of equal weights $c_\ell=1$, Theorems \ref{main regularity theorem delta positive all of space} and \ref{main regularity theorem delta positive} can be found in \cite{BrakkeMorgan}; see also the paper \cite{MH} for methods of existence and regularity.
\end{remark}

To state our asymptotic resolution theorem on the ball, we require some knowledge of the regularity for minimizers of the $\delta=0$ problem. In the general immiscible fluids problem, there may be singular points where more than three chambers meet; see \cite[Figure 1.1]{Chan}, \cite[Figure 7]{FGMMNPY}. Since we are interested in triple junction singularities, below is a description of the behavior of minimizers on the ball in some cases where all singularities are triple junctions. 

\begin{theorem}[Regularity on the Ball for $\delta=0$]\label{main regularity theorem delta zero}
    If $N=3$ or $c_\ell=1$ for $0\leq \ell \leq N$ and $\S^0$ is a minimizer for $\mathcal{F}$ among $\mathcal{A}_0^h$, then every connected component of $\partial S_\ell^0 \cap \partial S_m^0 \cap B$ for non-zero $\ell$ and $m$ is a line segment terminating at an interior triple junction $x\in \partial S_\ell^0 \cap \partial S_m^0 \cap \partial S_n^0 \cap B$, at $x\in \partial S_{\ell}^0 \cap \partial S_m^0 \cap \partial B$ which is a jump point of $h$, or at a boundary triple junction $x\in \partial S_{\ell}^0 \cap \partial S_m^0 \cap \partial S_n^0\cap \partial B$ which is a jump point of $h$. Moreover, for each triple $\{\ell,m,n\}$ of distinct non-zero indices there exists angles $\theta_\ell$, $\theta_m$, $\theta_n$ satisfying
\begin{align}\label{classical angle conditions intro}
    \frac{\sin \theta_\ell}{c_{m}+c_n}=\frac{\sin \theta_m}{c_{\ell}+c_n}=\frac{\sin \theta_n}{c_{\ell}+c_m}
\end{align}
such that if $x\in B$ is an interior triple junction between $S_\ell^0$, $S_m^0$, and $S_m^0$, then there exists 
$r_x>0$ such that $S_\ell^0\cap B_{r_x}(x)$ is a circular sector determined by $\theta_\ell$, and similarly for $m$, $n$. Finally, any connected component of $S_\ell^0$ for $1\leq \ell\leq N$ is convex.
\end{theorem}

\begin{remark}
    The proof of Theorem \ref{main regularity theorem delta zero} also applies when $N>3$ or $c_\ell$ are merely positive to show that the interfaces of a minimizer are finitely many segments meeting at isolated points. For the immiscible fluids problem on the ball, this has been observed in \cite[Corollary 4.6]{Morgan98}; see also \cite{White85}. Therefore, one may prove Theorem \ref{main regularity theorem delta zero} by classifying the possible tangent cones if $N=3$ or $c_\ell=1$ (Theorem \ref{zero delta classification theorem}). Since the proof of Theorem \ref{main regularity theorem delta positive}, which is in the language of sets of finite perimeter, can be easily modified to include a full proof of Theorem \ref{main regularity theorem delta zero}, we provide these arguments for completeness and as an alternative to the approach in \cite{Morgan98} via rectifiable chains.  
\end{remark}

\noindent Our last main result is a complete resolution of minimizers on the ball for small $\delta$ and equal weights.

\begin{theorem}[Resolution for Small $\delta$ on the Ball]\label{resolution for small delta corollary}
Suppose that $c_\ell=1$ for $0\leq \ell\leq N$ and $h\in BV(\partial B;\{1,\dots,N\})$. Then there exists $\delta_0>0$, a function $f(\delta)\to 0$ as $\delta \to 0$, and $r>0$, all depending on $h$, such that if $0<\delta<\delta_0$ and $\S^\delta$ is a minimizer in \eqref{ball problem}, then $|G^\delta|=\delta$ and there exists a minimizer $\S^0$ among $\mathcal{A}_{0}^h$ such that 
\begin{align}\label{hausdorff convergence of chambers in corollary}
\max \big\{ \sup\{\dist (x,S_\ell^0):x\in S_\ell^{\delta}\}\,,\,\sup\{\dist (x,S_\ell^{\delta}):x\in S_\ell^{0}\} \big\} \leq f(\delta)\quad \textit{for }1\leq \ell \leq N
\end{align}
and, denoting by $\Sigma$ the set of interior and boundary triple junctions of $\S^0$,
\begin{align}\label{hausdorff convergence of remnants in corollary}
\max \big\{ \sup \{\dist (x,\Sigma):x\in S_\ell^{0}\}\,,\,\sup \dist \{(x,G^{\delta}):x\in \Sigma\} \big\} \leq f(\delta)
\end{align}
and for each $x\in \Sigma$, $B_{r}(x) \cap \partial G^{\delta}$ consists of three circle arcs of curvature $\kappa=\kappa(\S^\delta)$.
\end{theorem}

\begin{remark}[Wetting of Singularities]\label{wetting remark}
    For the soap bubble capillarity analogue of \eqref{thin film cluster formulation} on $B$,
\begin{align}\label{thin film ball formulation}
    \inf\{ \mathcal{F}(\S) : \S\in \mathcal{A}_\delta^{h},\, \mbox{$S_\ell$ open, $\cl S_\ell \cap \cl S_m\subset \{x\in \partial B: \mbox{$h$ jumps between 
 }\ell, m\}$} \}\,,
\end{align}
we may also use Theorem \ref{main regularity theorem delta positive} to approximate a minimizer in \eqref{ball problem} by a sequence satisfying the restrictions in \eqref{thin film ball formulation}. Therefore, if $\delta>0$ is small, a minimizing sequence for \eqref{thin film ball formulation} converges to a minimizer $\S^\delta$ of \eqref{ball problem}, which in turn is close to a minimizer $\S^0$ for the $\delta=0$ problem. Furthermore, by Theorem \ref{resolution for small delta corollary}, if $\delta<\delta_0$ and the weights $c_\ell$ are equal, each singularity of $\S^0$ is ``wetted" by a component of $G^\delta$ bounded by three circular arcs; see Figure \ref{perturbation figure}. Also, \eqref{hausdorff convergence of remnants in corollary} shows that $\Sigma$ coincides with the set of accumulation points of the ``wet" regions $G^\delta$ as $\delta \to 0$. In the context of the Plateau problem in $\mathbb{R}^2$, this equivalence has been conjectured in \cite[Remark 1.7]{KMS3}. If $\dist(\Sigma,\partial B)>f(\delta)$, then $\S^\delta$ coincides with the ``wetting" of a dry minimizer; see Remark \ref{remark:explicit description remark}.
\end{remark}

\begin{remark}[Triple Junctions for Vector Allen-Cahn]
    Theorem \ref{main regularity theorem delta positive} is used in a construction by \'{E}. Sandier and P. Sternberg of an entire solution $U:\mathbb{R}^2\to \mathbb{R}^2$ to the system $\Delta U = \nabla_u W(U) $ for a triple-well potential $W$ without symmetry assumptions on the potential \cite{SS}.
\end{remark}

\subsection{Idea of proof} The outline to prove Theorems \ref{main regularity theorem delta positive all of space} and \ref{main regularity theorem delta positive} can be summarized in two main steps: first, classifying the possible blow-ups at any interfacial point of a minimizer $\S^\delta$, and; second, using one of the (a priori non-unique) blow-ups at $x$ to resolve $\S^\delta$ in a small neighborhood of $x$. To demonstrate the ideas, we describe these steps for a minimizer $\S^\delta$ for the problem \eqref{all of space problem} on $\mathbb{R}^2$ at $x=0$. For the classification of blow-ups, we use a blow-up version of the observation below \eqref{trace constraint} to show that no blow-up of any chamber $S_\ell^\delta$ can be anything other than a halfspace. This of course differs from the usual blow-ups in two-dimensional cluster problems, in which three or more chambers can meet at a point.
\par
Armed now with a list of the possible blow-ups at $0$, which we do not yet know are unique, we must use them to resolve the minimizer in a small neighborhood of $0$. In the case that there exists a blow-up coming from $G^\delta$ and a single chamber $S_\ell^\delta$, lower area density estimates on the remaining chambers imply that in a small ball $B_r(0)$, $ S_{\ell'}^\delta \cap B_r(0)=\emptyset$ for $\ell \neq \ell'$, so that $\partial S_\ell^\delta \cap B_r(0)$ is regular by the classical theory for volume-constrained perimeter minimizers. The main hurdle is when the blow-up at $0$ is two halfspaces coming from $S_{\ell_i}^\delta$ for $i=1,2$. In the classical regularity theory for planar clusters (see \cite[Section 11]{White} or \cite[Corollary 4.8]{Leo01}), this would imply that on $B_r(0)$, the interface must be an arc of constant curvature separating each $S_{\ell_i}^\delta \cap B_r(0)$. Here, there is the possibility that $0\in \partial G^\delta$ but $G^\delta$ has density $0$ at $0$. This behavior cannot be detected at the blow-up level, although one suspects the interfaces near $0$ should be two ordered graphs over a common line which coincide at $0$ and possible elsewhere also. To prove this and thus complete the local resolution, we use the convergence along a sequence of blow-ups to a pair of halfspaces and the density estimates on the other chambers to locate a small rectangle $Q=[-r,r]\times [r,r]$ such that $Q\subset S_{\ell_1}^\delta \cup S_{\ell_2}^\delta \cup G^\delta$ and $\partial Q \cap \partial S_{\ell_i}^\delta =\{(-r,a_i), (r,b_i)\}$ for some $a_1\leq a_2$ and $b_1\leq b_2$. At this point, since we have the desired graphicality on $\partial Q$, we can combine a symmetrization inequality for sets which are graphical on the boundary of a cube (Lemma \ref{symmetrization lemma}), the minimality of $\S^\delta$, and the necessary conditions for equality in Lemma \ref{symmetrization lemma} to conclude that $\partial S_{\ell_i}^\delta \cap Q$ are two ordered graphs. 

\subsection{Organization of the paper} In Section \ref{sec:notation and prelim}, we recall some preliminary facts. Next, we prove the existence of minimizers in Section \ref{sec:existence of minmizers}. Section \ref{sec:existence of blowup cones} contains the proof of the existence and classification of blow-up cones at any interfacial point. 
In Sections \ref{sec: proofs of main regularity theorems} and \ref{sec:proof of delta 0 theorem}, we prove Theorems \ref{main regularity theorem delta positive all of space} and \ref{main regularity theorem delta positive} and Theorem \ref{main regularity theorem delta zero}, respectively. Finally, in Section \ref{sec:resolution for small delta}, we prove Theorem \ref{resolution for small delta corollary}.
\subsection{Acknowledgments} This work was supported by the NSF grant RTG-DMS 1840314. I am grateful to \'{E}tienne Sandier and Peter Sternberg for several discussions during the completion of this work and to Frank Morgan for valuable comments on the literature for such problems.

\section{Notation and Preliminaries}\label{sec:notation and prelim}
\subsection{Notation}
Throughout the paper, $B_r(x)=\{y\in \mathbb{R}^2:|y-x|<r\}$. When $x=0$, we set $B_R:=B_R(0)$ and $B=B_1(0)$. Also, for any Borel measurable $U$, we set
\begin{align*}
    \mathcal{F}(\S;U) = \sum_{\ell=0}^N c_\ell P(S_\ell;U)\,.
\end{align*}
We will use the notation $E^{(t)}$ for the points of Lebesgue density $t\in [0,1]$.
\par
We remark that since $h\in BV(\partial B; \{1,\dots,N\})$, there exists a partition of $\partial B$ into $N$ pairwise disjoint sets $\{A_1,\dots, A_N\}$ such that $h = \sum_{\ell=1}^N \ell \,1_{A_\ell}$, and each $A_\ell$ is a finite union of pairwise disjoint arcs:
\begin{align}\label{arc definition}
A_\ell := \cup_{i=1}^{I_\ell}a_i^\ell\,.
\end{align}
For each $1\leq \ell \leq N$ and $1\leq i \leq I_\ell$, we let
\begin{align}\label{chord def}
c_i^\ell
\end{align}
be the chord that shares endpoints with $a_i^\ell$. Finally, we call
\begin{align}\label{circular segments}
C_i^\ell
\end{align}
the open circular segments (regions bounded by an arc and its corresponding chord) corresponding to the pair $(a_i^\ell, c_i^\ell)$. 

\subsection{Preliminaries}
Regarding the functional $\mathcal{F}$, we observe that when $\delta=0$, 
\begin{align}\notag
    \mathcal{F}(\S) = \sum_{0\leq \ell< m\leq N} c_{\ell m}\mathcal{H}^1(\partial^* S_\ell \cap \partial^* S_m)\,,
\end{align}
where $c_{\ell m } :=c_\ell + c_m$, and the positivity of $c_\ell$ for $1\leq \ell \leq N$ is equivalent to the strict triangle inequalities
\begin{align}\label{strict triangle ineq}
    c_{\ell m}< c_{\ell i} + c_{i m}\quad \forall \ell\neq m\neq i \neq \ell\,.
\end{align}
We also note that for any $h\in BV(\partial B;\{1,\dots,N\})$, the energy of any cluster $\S$ satisfying the boundary condition \eqref{trace constraint} can be decomposed as
\begin{align}\label{energy independent of boundary}
    \mathcal{F}(\S) = 2\pi c_0 +\sum_{\ell=1}^N c_\ell \mathcal{H}^1(A_\ell)+ \sum_{\ell=1}^N c_\ell P(S_\ell;B) =: C(h) + \mathcal{F}(\S;B)\,,
\end{align}
where $C(h)$ is a constant independent of $\S$. Therefore, minimizing $\mathcal{F}$ among $\mathcal{A}_\delta^h$ for any $\delta>0$ is equivalent to minimizing $\mathcal{F}(\cdot;B)$, so we will often ignore the boundary term for the problem on the ball.\par
We now recall some facts regarding sets of finite perimeter. Unless otherwise stated, we will always adhere to the convention that among the Lebesgue representatives of a given set of finite perimeter $E$, we are considering one 
that satisfies \cite[Proposition 12.19]{Mag12}
\begin{align}\label{boundary convention 1}
    \spt \,\mathcal{H}^1 \mres \partial^* E = \partial E
\end{align}
and
\begin{align}\label{boundary convention 2}
\partial E= \{x:0<|E \cap B_r(x)|<\pi r^2\,\,\forall r>0 \}\,.
\end{align}



\noindent We will need some facts regarding slicing sets of finite perimeter by lines or circles.

\begin{lemma}[Slicing sets of finite perimeter]\label{slicing lemma}
    Let $u(x)=x\cdot \nu$ for some $\nu\in \mathbb{S}^1$ or $u(x)=|x-y|$ for some $y\in \mathbb{R}^2$, and, for any set $A$, let $A_t$ denote $A\cap \{u=t\}$. Suppose that $E\subset \mathbb{R}^2$ is a set of finite perimeter.
\begin{enumerate}[label=(\roman*)]
\item For every $t\in \mathbb{R}$, there exist traces $E^+_t$, $E^-_t\subset \{u=t \}$ such that
\begin{align}\label{trace difference}
    \int_{\{u=t\}}|\mathbf{1}_{E^+_t}-\mathbf{1}_{E^-_t}|\,d\mathcal{H}^1 = P(E; \{u=t \})\,.
\end{align}
\item Letting $S=\{x:x\cdot \nu^\perp \in [a,b]\}$ for compact $[a,b]$ when $u=x\cdot \nu$ or $S = \mathbb{R}^2$ when $u=|x-y|$,
\begin{align}\label{convergence of trace integrals}
    \lim_{s\downarrow t} \int_{\{u=s\}\cap S} \mathbf{1}_{E^-_t}\,d\mathcal{H}^1 = \int_{\{u=t\}\cap S} \mathbf{1}_{E^+_t}\,d\mathcal{H}^1\,.
\end{align}
\item For almost every $t\in \mathbb{R}$, $E_t^+=E_t^-=E_t$ up to an $\mathcal{H}^1$-null set, $E_t$ is a set of finite perimeter in $\{u=t\}$, and
\begin{align}\label{equivalence of boundaries}
    \mathcal{H}^{0}((\partial^* E)_t \Delta \partial^*_{\{u=t\}}E_t)=0\,.
\end{align}
\end{enumerate}
\end{lemma}
\begin{proof}
    The first item can be found in \cite[(2.15)]{Giu84}. We prove the second item when $u=\vec{e}_1 \cdot x$; the proof with any other $\nu$ or when $u=|x-y|$ is similar. By the divergence theorem \cite[Theorem 2.10]{Giu84},
\begin{align}\notag
  0 &= \int_{(t,s)\times (a,b) \cap E}\dive \vec{e}_1  \\ \notag
  &=\int_{\{u=s\}\cap S} \mathbf{1}_{E^-_t}\,d\mathcal{H}^1 - \int_{\{u=t\}\cap S} \mathbf{1}_{E^+_t}\,d\mathcal{H}^1 + \int_{\partial^* E \cap (t,s)\times (a,b)} \vec{e}_1 \cdot \nu_{E}\,d\mathcal{H}^1 \\ \notag
  &\qquad + \int_{\partial^* (E \cap(t,s)\times (a,b)) \cap (t,s)\times\{a,b\}} \vec{e}_1 \cdot \nu_{E \cap(t,s)\times (a,b)}\,d\mathcal{H}^1\,.
\end{align}
Now the last term on the right hand side is bounded by $2(s-t)$ and vanishes as $s\to t$. Also, the third term on the right hand side is bounded by $P(E;(t,s)\times (a,b))$, which vanishes as $s\to t$ since $(t,s)\times (a,b)$ is a decreasing family of bounded open sets whose intersection is empty and $B\to P(E; B)$ is a Radon measure. The limit \eqref{convergence of trace integrals} follows from letting $s$ decrease to $t$.
\par
Moving on to $(iii)$, we recall that for $\mathcal{H}^1$-a.e. $x\in \{u=t\}\cap E_t^+$,
\begin{align}\label{averaging prop of traces}
    1=\lim_{r\to 0}\frac{|B_r(x) \cap E \cap \{u>t \}|}{\pi r^2/2}
\end{align}
and similarly for $E_t^-$ \cite[2.13]{Giu84}. Next, by \eqref{trace difference}, 
    \begin{align}\label{agreement of plus minus}
    \mathcal{H}^1(E_t^+\Delta E_t^-)=0 \quad\textup{ if }P(E;\{u=t\})=0\,,
    \end{align}
which is all but at most countably many $t$. Now, for any $x \in \{ u=t\}$ that is also a Lebesgue point of $E$,
\begin{align}\label{equalities at Lebesgue points}
    1= \lim_{r\to 0}\frac{|B_r(x) \cap E|}{\pi r^2} = \lim_{r\to 0}\frac{|B_r(x) \cap E \cap \{u>t \}|}{\pi r^2/2}=\lim_{r\to 0}\frac{|B_r(x) \cap E \cap \{u<t \}|}{\pi r^2/2}\,.
\end{align}
Since $\mathcal{L}^2$-a.e. $x\in E$ is a Lebesgue point, we conclude from \eqref{averaging prop of traces}, \eqref{agreement of plus minus}, and \eqref{equalities at Lebesgue points} that $\mathcal{H}^1(E_t \Delta E_t^\pm)=0$ for $\mathcal{H}^1$-a.e. $t$. Lastly, \eqref{equivalence of boundaries} when slicing by lines can be found in \cite[Theorem 18.11]{Mag12} for example. The case of slicing by circles follows from the case of lines and the fact that smooth diffeomorphisms preserve reduced boundaries \cite[Lemma A.1]{KinMagStu22}.
\end{proof}

\noindent We will use the following fact regarding the intersection of a set of finite perimeter with a convex set. 

\begin{lemma}\label{convexity lemma}
    If $E$ is a bounded set of finite perimeter and $K$ is a convex set, then
\begin{align}\notag
    P(E \cap K) \leq P(E)\,,
\end{align}
with equality if and only if $|E\setminus  K|=0 $.
\end{lemma}

\begin{proof}
    The argument is based on the facts that the intersection of such $E$ with a halfspace $H$ decreases perimeter (with equality if and only $|H\setminus E|=0$) and any convex set is an intersection of halfspaces. We omit the details. 
\end{proof}

\noindent Our last preliminary regarding sets of finite perimeter is a symmetrization inequality, which for convenience, we state in the setting it will be employed later. 

\begin{lemma}\label{symmetrization lemma}
    Let $Q'=[t_1,t_2]\times [-1,1]$. Suppose that $E\subset Q'$ is a set of finite perimeter such that $(t_1,t_2)\times (-1, -1/4)\subset E^{(1)}\subset (t_1,t_2)\times (-1,1/4)$ and, for some $a_1,a_2\in [-1/4,1/4]$,
\begin{align}\label{trace assumption}
    E_{t_1}^+=[-1,a_1]\,,\quad E_{t_2}^-=[-1,a_2]\quad\textit{up to $\mathcal{H}^1$-null sets}\,,
\end{align}
where $E_{t_1}^+$, $E_{t_2}^-$, viewed as subsets of $\mathbb{R}$, are the traces from the right and left, respectively, slicing by $u(x)=x\cdot e_1$. Then the set $E^h = \{(x_1,x_2):-1 \leq x_2 \leq \mathcal{H}^1(E_{x_1})-1 \}$ satisfies $|E^h|=|E|$,
\begin{align}\label{preserves traces}
    (E^h)^+_{t_1}=[-1,a_1]\,,\quad (E^h)^-_{t_2}=[-1,a_2]\quad\textit{up to $\mathcal{H}^1$-null sets}
\end{align}
and
\begin{align}\label{symmetrization inequality}
    P(E^h;\mathrm{int}\,Q') \leq P(E;\mathrm{int}\,Q')\,.
\end{align}
Moreover, if equality holds in \eqref{symmetrization inequality}, then for every $t\in (t_1,t_2)$, $(E^{(1)})_t$ is an interval. 
\end{lemma}

\begin{remark}
    The superscript $h$ is for ``hypograph."
\end{remark}

\begin{figure}
\begin{overpic}[scale=0.5,unit=1mm]{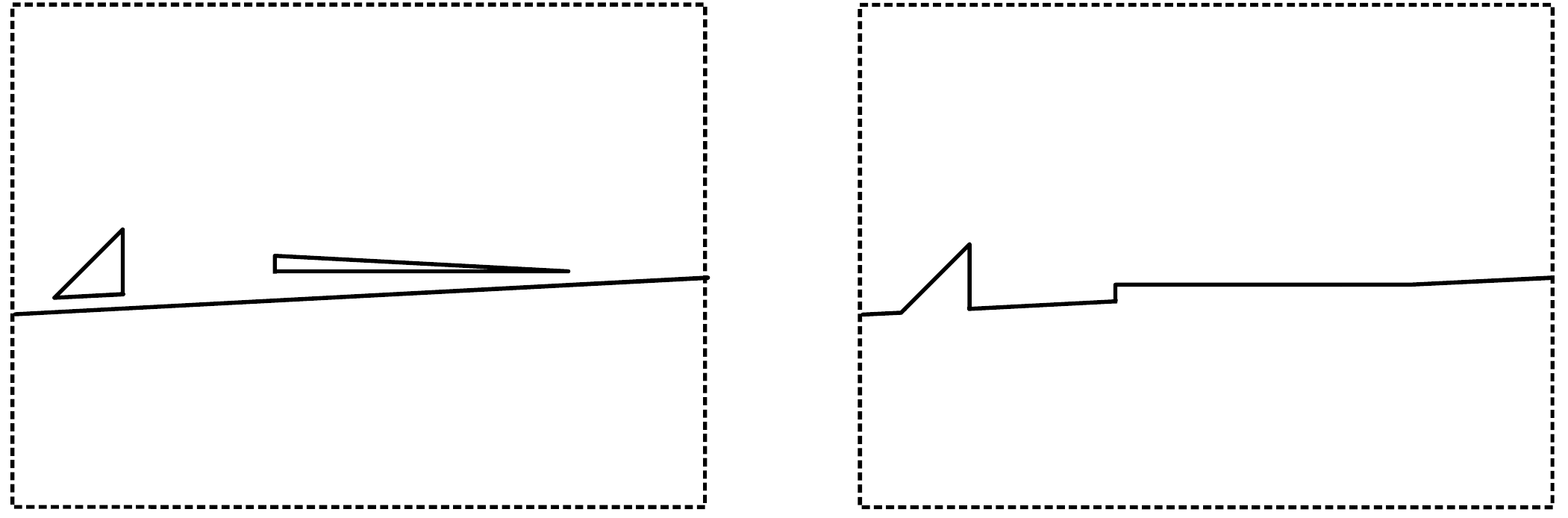}
    \put(10,10){\small{$E$}}
    \put(70,10){\small{$E^h$}}
    \put(-5,25){\small{$Q'$}}
    \put(101,25){\small{$Q'$}}
\end{overpic}
\caption{Both the sets $E$ and $E^h$ have the same trace on $\partial Q'$, and $P(E^h;\mathrm{int}\,Q')<P(E;\mathrm{int}\,Q')$ because $E$ has vertical slices which are not intervals.}\label{symm figure}
\end{figure}
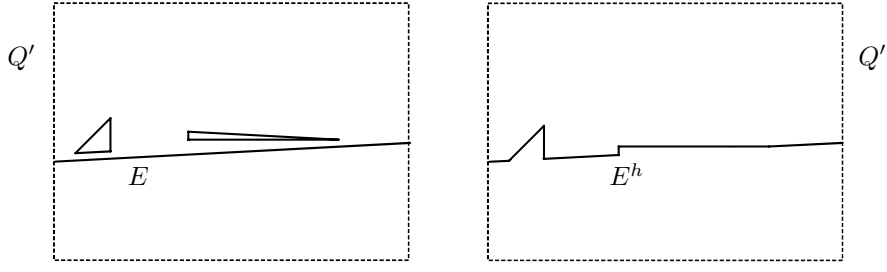

\begin{proof}
The preservation of area $|E^h|=|E|$ is immediate by Fubini's theorem, so we begin with the first equality in \eqref{preserves traces}, and the second is analogous. We recall from \eqref{averaging prop of traces} that for $\mathcal{H}^1$-a.e. $x\in \{t_1\}\times [-1,1]\cap (E^h)_{t_1}^+$,
\begin{align}\label{averaging prop of traces 2}
    1=\lim_{r\to 0}\frac{|B_r(x) \cap E^h \cap Q'|}{\pi r^2/2}\,.
\end{align}
From this property and the fact that the vertical slices of $E^h$ are intervals of height at least $3/4$, it follows that $(E^h)_{t_1}^+$ is $\mathcal{H}^1$-equivalent to an interval $[-1,a]$ for some $a\geq -1/4$. Furthermore, $a=a_1$ is a consequence of \eqref{convergence of trace integrals} and the fact that the rearrangement $E^h$ preserves the $\mathcal{H}^1$-measure of each vertical slice:
\begin{align*}
    a_1 = \int_{\{t_1\}\times [-1,1]} \mathbf{1}_{E_{t_1}^+}\,d\mathcal{H}^1 &= \lim_{s\downarrow t_1}\int_{\{s\}\times [-1,1]} \mathbf{1}_{E_{s}^-}\,d\mathcal{H}^1\\
    &=\lim_{s\downarrow t_1}\int_{\{s\}\times [-1,1]} \mathbf{1}_{(E^h)_{s}^-}\,d\mathcal{H}^1=\int_{\{t_1\}\times [-1,1]} \mathbf{1}_{(E^h)_{t_1}^+}\,d\mathcal{H}^1=a\,.
\end{align*}
\par
Moving on to \eqref{symmetrization inequality}, let consider the sets $E^r$ which is the reflection of $E$ over $\{x_2=-1\}$, and $G=E \cup E^r$.  We denote by the superscript $s$ the Steiner symmetrization of a set over $\{x_2=-1\}$. We note that
\begin{align}\notag
    G^s \cap Q' = E^h\,.
\end{align}
Since $(t_1,t_2)\times (-1, -1/4)\subset E^{(1)}\subset (t_1,t_2)\times (-1,1/4)$ and Steiner symmetrizing decreases perimeter, we therefore have 
\begin{align}\notag
    P(E; \mathrm{int}\,Q') = \frac{P(G;\{x_1\in (t_1,t_2)\})}{2} \geq \frac{P(G^s;\{x_1\in (t_1,t_2)\})}{2}=P(G^s;\mathrm{int}\,Q')=P(E^h;\mathrm{int}\,Q')\,,
\end{align}
which is \eqref{symmetrization inequality}. Furthermore, equality can only hold if every almost every vertical slice of $G$ is an interval, which in turn implies that $E_t$ is an interval for almost every $t\in (t_1,t_2)$. By \cite[Lemma 4.12]{Fus04}, every slice $(E^{(1)})_t$ is an interval.
\end{proof}

We conclude the preliminaries with a lemma regarding of the convergence of convex sets.

\begin{lemma}\label{convex sets convergence lemma}
    If $\{C_n\}$ is a sequence of equibounded, compact, and convex sets in $\mathbb{R}^n$, then there exists compact and convex $C \subset\mathbb{R}^n$ such that $\mathbf{1}_{C_n}\to \mathbf{1}_C$ almost everywhere and
\begin{align}\label{hausdorff convergence of convex sets}
\max \big\{ \sup_{x\in C_n} \dist (x,C)\,,\,\sup_{x\in C} \dist (x,C_n) \big\} \to 0 \,.
\end{align}
\end{lemma} 

\begin{proof}
By the Arzel\'{a}-Ascoli Theorem, there exists a compact set $C\subset\mathbb{R}^n$ such that $\dist(\cdot, C_n)\to \dist(\cdot,C)$ uniformly. Therefore, $C_n\to C$ in the Kuratowski sense \cite[Section 2]{FFLM}, $C$ is convex, and $\mathbf{1}_{C_n}\to \mathbf{1}_C$ almost everywhere \cite[Remark 2.1]{FFLM}. Since $C_n$ are equibounded and $C$ is compact, the Kuratowski convergence is equivalent to Hausdorff convergence, which is \eqref{hausdorff convergence of convex sets}.\end{proof}

\section{Existence of Minimizers}\label{sec:existence of minmizers}

First we establish the existence of minimizers for the problem \eqref{ball problem} on the ball. A byproduct of the proof is a description of minimizers on each of the circular segments from \eqref{circular segments}; see Fig. \ref{polygon figure}.

\begin{theorem}[Existence on the ball]\label{existence theorem}
For any $\delta \geq 0$ and $h\in BV(\partial B;\{1,\dots,N\})$, there exists a minimizer of $\mathcal{F}$ among the class $\mathcal{A}_\delta^h$. Moreover, any minimizer $\S^\delta$ for $\delta \geq 0$ satisfies
\begin{align}\label{circular segment containment}
\cup_{i=1}^{I_\ell} C_i^\ell \subset S_\ell^{(1)}\quad\textit{for each }1\leq \ell \leq N\,.
\end{align}
\end{theorem}


\begin{proof}The proof is divided into two steps. The closed convex sets
\begin{align*}
    K_\ell:= \cl (B \setminus (\cup_{i=1}^{I_\ell} C_i^\ell ))\,, \quad 1\leq \ell \leq N
\end{align*}
will be used throughout.\par
\medskip
\noindent\textit{Step one}: First we show that given any $\S\in \mathcal{A}_\delta^h$, the cluster $\tilde{\S}$ defined via
\begin{align*}
    \tilde{S}_\ell &:= \Big(S_\ell \cap \bigcap_{j\neq \ell} K_j \Big)\cup \bigcup_{i=1}^{I_\ell} C_i^\ell \quad 1\leq \ell \leq N\,,\qquad \tilde{S}_0=B^c\,, \qquad \tilde{G}=(\tilde{S}_0\cup \dots \cup \tilde{S}_N)^c
\end{align*}
satisfies $\tilde{\S}\in \mathcal{A}_\delta^h$ and
\begin{align}\label{less energy}
    \mathcal{F}(\tilde{\S}) \leq \mathcal{F} (\S)\,,
\end{align}
\begin{figure}
\begin{overpic}[scale=0.3,unit=1mm]{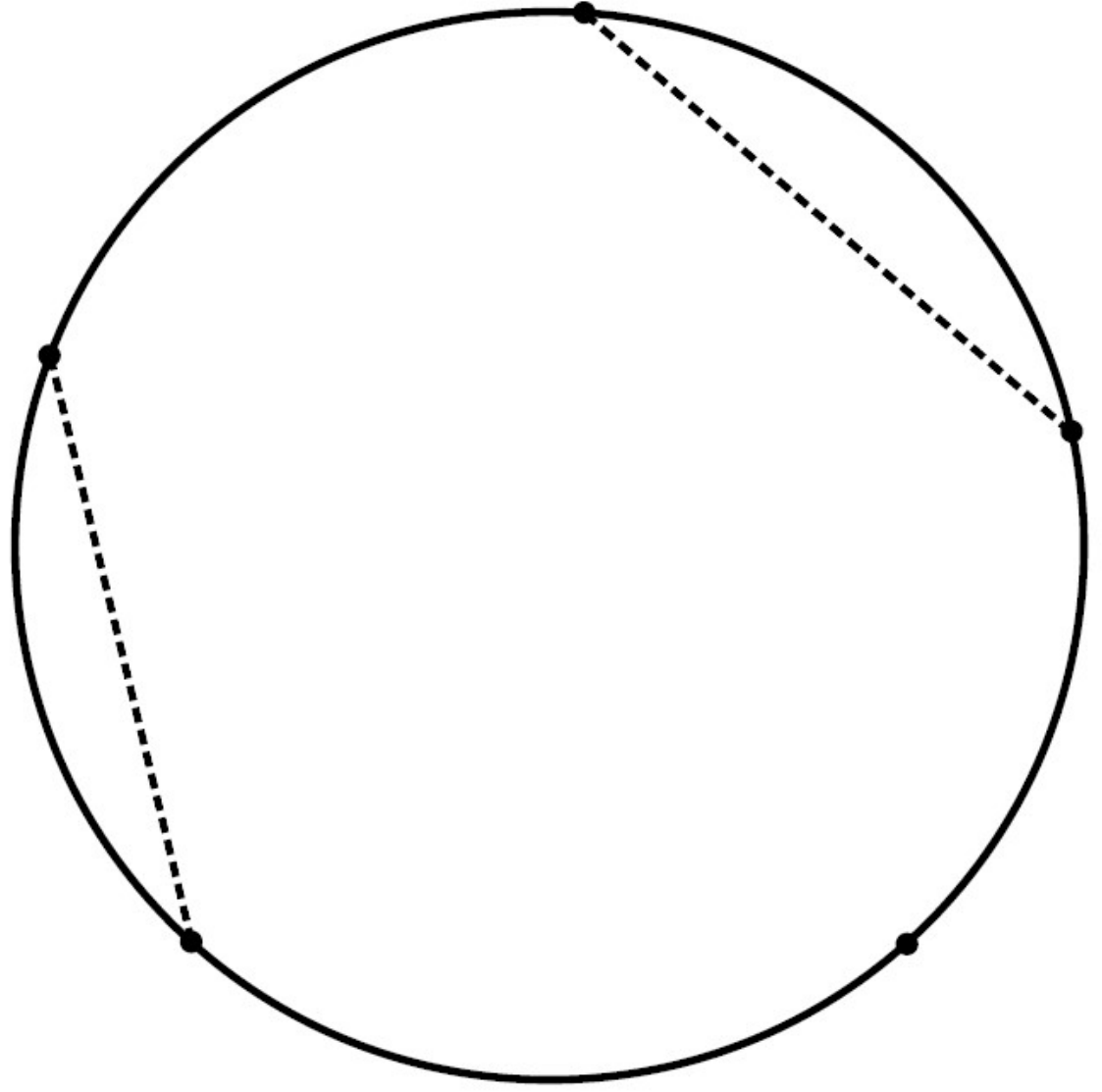}
    \put(-8,33){\small{$a_1^1$}}
    \put(4,10){\small{$x_1$}}
    \put(86,10){\small{$x_2$}}
    \put(100,55){\small{$x_3$}}
    \put(-10,62){\small{$x_5$}}
    \put(48,102){\small{$x_4$}}
    \put(85,84){\small{$a_2^1$}}
    \put(12,40){\small{$c_1^1$}}
    \put(64,72){\small{$c_2^1$}}
\end{overpic}
\caption{Here $h$ jumps at $x_i$, $1\leq i\leq 5$, and $\{h=1\}$ on the arcs $a_1^1$ and $a_2^1$. Equation \eqref{circular segment containment} states that for any minimizer $\S$ with this boundary data, $S_1^{(1)}$ must contain the regions bounded by $a_j^1$ and the chords $c_j^1$ for $j=1,2$.}\label{polygon figure}
\end{figure}
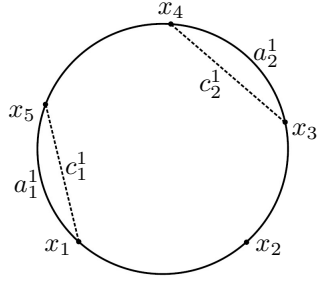
with equality if only if
\begin{align}\label{containment}
\cup_{i=1}^{I_\ell} C_i^\ell \subset S_\ell^{(1)}\quad\quad  \forall 1\leq \ell \leq N\,.
\end{align}
The proof relies on Lemma \ref{convexity lemma}, which states that if $E$ is a set of finite perimeter, $|E|<\infty$, and $K$ is a closed convex set, then $E \cap K$ is a set of finite perimeter and
\begin{align}\label{convexity inequality}
    P(E \cap K) \leq P(E)\,,
\end{align}
with equality if and only if $|E\setminus K|=0$. 
For given $\S \in \mathcal{A}_\delta^h$, let us first consider the cluster $\S'$, where
\begin{align*}
    S_1' := S_1 \cup \bigcup_{i=1}^{I_1} C_i^1\,,\,\, S_\ell' := S_\ell \cap K_1\,,\,\, 2\leq \ell \leq N\qquad S'_0=B^c\,, \qquad G'=(S_0'\cup \dots \cup S_N')^c\,.
\end{align*}
By the trace condition \eqref{trace constraint} and the definition of $S_\ell'$,
\begin{align}\label{trace constraint 2}
    S_\ell' \cap \partial B = \{x\in \partial B : h(x) = \ell \}\textup{ for $1\leq \ell \leq N$ in the sense of traces}\,.
\end{align}
Also, since $ G'= B \setminus \cup_\ell S_\ell'$ satisfies
\begin{align*}
    | G'| =  |(B\cap K_1) \setminus \cup_\ell S_\ell| \leq |B \setminus \cup_\ell S_\ell|\leq \delta\,,
\end{align*}
we have
\begin{equation}\label{remains in A delta}
\S' \in \mathcal{A}_\delta^h\,.
\end{equation}
Now for $2\leq \ell\leq N$, we use \eqref{convexity inequality} to estimate
\begin{align}\label{23 estimate}
   c_\ell P(S_\ell)
    \geq c_\ell P(S_\ell \cap K_1) =
     c_\ell P(S_\ell')\,.
\end{align}
For $\ell=1$, we first recall the fact for any set of finite perimeter $E$,
\begin{equation}\label{complement equality}
P(E;B) = P(E^c ; B)\,.
\end{equation}
Applying \eqref{complement equality} with $S_1$, then \eqref{trace constraint}, \eqref{convexity inequality}, and \eqref{trace constraint 2}, and finally \eqref{complement equality} with $S_1'$, we find that
\begin{align}\notag
   P(S_1;B) &= P((\cup_{\ell=2}^N S_\ell \cup G) ; B) \\ \notag
    &= P(\cup_{\ell=2}^N S_\ell \cup G) - \mathcal{H}^1(\cup_{\ell=2}^N A_\ell) \\ \notag
    &\geq P((\cup_{\ell=2}^N S_\ell \cup G) \cap K_1) - \mathcal{H}^1(\cup_{\ell=2}^N A_\ell) \\ \notag
    &= P(\cup_{\ell=2}^N S_\ell' \cup G'; B) \\ \label{1 estimate}
    &= P(S_1' ; B)\,.
\end{align}
Adding $\mathcal{H}^1(A_\ell)$ to \eqref{1 estimate}, multiplying  
by $c_1$, and combining with \eqref{23 estimate} gives
\begin{align}\label{first projection inequality}
    \mathcal{F}(\S)=\sum_{\ell=0}^N c_\ell P(S_\ell) \geq \sum_{\ell=0}^N c_\ell P(S_\ell')=\mathcal{F} (\S')\,,
\end{align}
and so we have a new cluster $\S'$, belonging to $\mathcal{A}_\delta^h$ by \eqref{remains in A delta}, that satisfies
\begin{align}\label{first step containment}
    \cup_{i=1}^{I_1} C_i^1 \subset (S_1')^{(1)}\,.
\end{align}
Repeating this argument for $2\leq\ell\leq N$ yields $\tilde{\S}\in\mathcal{A}_\delta^h$ satisfying \eqref{less energy} as desired. Turning now towards the proof that equality in \eqref{less energy} implies \eqref{containment}, we prove that the containment for $\ell=1$ in \eqref{containment} is entailed by equality; the other $N-1$ implications are analogous. If \eqref{less energy} holds as an equality, then \eqref{23 estimate} and \eqref{1 estimate} must hold as equalities as well. But by the characterization of equality in \eqref{convexity inequality}, this can only hold if $(\cup_{\ell=2}^N S_\ell \cup G) \cap K_1 = \cup_{\ell=2}^N S_\ell \cup G$, which yields the first containment in \eqref{containment}.\par
Finally, let us also remark that an immediate consequence of this step is that if a minimizer of $\mathcal{F}$ exists among $\mathcal{A}_\delta^h$, then \eqref{circular segment containment} must hold. It remains then to prove the existence of a minimizer.\par
\medskip
\noindent\textit{Step two}:
Let $\{\S^m\}_m$ be a minimizing sequence of clusters for $\mathcal{F}$ among $\mathcal{A}_\delta^h$ (the infimum is finite). Due to the results of step one, we can modify our minimizing sequence so that
\begin{align}\label{segment containment for existence}
    \cup_{i=1}^{I_\ell} C_i^\ell \subset (S_{\ell}^m)^{(1)}\quad \forall m\,,\,\, \forall 1\leq \ell \leq N
\end{align}
while also preserving the asymptotic minimality of the sequence. By compactness in $BV$ and \eqref{segment containment for existence}, after taking a subsequence, we obtain a limiting cluster $\S$ that satisfies the trace condition \eqref{trace constraint}, and, by lower-semicontinuity in $BV$, minimizes $\mathcal{F}$ among $\mathcal{A}_\delta^h$.  
\end{proof}

\begin{remark}[Existence of minimizer for a functional with boundary energy]
    One might also consider the minimizing the energy
\begin{align*}
    \mathcal{F}(\S;B) + \sum_{m=1}^N\sum_{\ell\neq m} (c_\ell + c_m)\mathcal{H}^1(\partial^* S_\ell \cap \{h=m\})\,,
\end{align*}
which penalizes deviations from $h$ rather than enforcing a strict trace condition, among the class
\begin{align}\notag
  \{(S_0,\dots,S_N,G) : |S_\ell \cap S_m|=0 \textup{ if $\ell \neq m$, }|G|\leq \delta,\, B^c = S_0\}\,.
\end{align}
For this problem, the same convexity-based argument as in step one of the proof of Theorem \ref{existence theorem} shows that in fact, minimizers exists and attain the boundary values $h$ $\mathcal{H}^1$-a.e. on $\partial B$. When $\delta=0$, this problem arises as the $\Gamma$-limit of a Modica-Mortola problem with Dirichlet condition \cite{SS}.
\end{remark}

Next, we prove existence for the problem on all of space. Since we are in the plane, the proof utilizes the observation that perimeter and diameter scale the same in $\mathbb{R}^2$. Existence should also hold in $\mathbb{R}^n$ for $n\geq 3$ using the techniques of \cite{Alm76}.

\begin{theorem}[Existence on $\mathbb{R}^2$]\label{existence on all of space}
    For any $\mm\in (0,\infty)^N$, there exists $R=R(\mm)$ such that for all $\delta \geq 0$, there exists a minimizer of $\mathcal{F}$ among the class $\mathcal{A}_\delta^{\mm}$ satisfying $\mathbb{R}^2\setminus B_R \subset S_0$. 
\end{theorem}

\begin{proof}
Let $\{\S^j\}_j\subset \mathcal{A}_\delta^{\mm}$ be a minimizing sequence with remnants $G^j$. The existence of a minimizer is straightforward if we can find $R>0$ such that, up to modifications preserving the asymptotic minimality, $B_R^c \subset S_0^j$ for each $j$. We introduce the sets of finite perimeter $E_j=\cup_{\ell=1}^N S_\ell^j \cup G^j$, which satisfy
    $P(E_j)\leq \max\{c_\ell^{-1}\}\mathcal{F}(\S^j)$ and $\partial^* E_j \subset \cup_{\ell=1}^N \partial^* S_\ell^j$. Decomposing $E_j$ into its indecomposable components $\{E^j_k\}_{k=1}^\infty$ \cite[Theorem 1]{AMMN01}, we have $\mathcal{H}^1(\partial^* E_k^j \cap \partial^* E_{k'}^j)=0$ for $k\neq k'$. Therefore, for the clusters $\S_k^j=((E_k^j)^c, S_1^j \cap E_k^j,\dots,S_N^j \cap E_k^j,G^j \cap  E_k^j )$,
\begin{align}\notag
    \mathcal{F}(\S^j) = \sum_{k=1}^\infty \mathcal{F}(\S_k^j)\,.
\end{align}
Furthermore, by the indecomposability of any $E_k^j$, there exists $x_k^j\in \mathbb{R}^2$ such that
\begin{align}\notag
 (G^j \cap E_k^j)\cup \cup_{\ell=1}^N S_\ell^j \cap E_k^j \subset  E_k^j \subset B_{P(E_k^j)}(x_k^j)\,.
\end{align}
By the uniform energy bound along the minimizing sequence and this containment, for any $j$, we may translate each $\mathcal{S}_k^j$ so that the resulting sequence of clusters satisfies $B_R^c \subset S_0^j$. Finally, we note that $R\leq 2\max\{c_\ell^{-1}\}\inf_{\mathcal{A}_\delta^{\mm}}\mathcal{F}$, and since that infimum is bounded independently of $\delta$, it depends only on $\mm$.
\end{proof}


\section{Existence and Classification of Blow-up Cones}\label{sec:existence of blowup cones}
In this section, we prove the existence of blow-up cones for minimizers and classify the possibilities. Since the proofs are mostly modified versions of standard arguments, we will often be brief in this section and describe the main ideas and adjustments. Also, we do not include any arguments for the case $\mathcal{A}_0^{\mm}$ as that regularity is known in $\mathbb{R}^2$ \cite{White85,Morgan98}.

\subsection{Perimeter-almost minimizing clusters} Lemma \ref{lambda r nought lemma} allows us to test minimality 
of $\S^\delta$ against competitors that do not satisfy the constraint required for membership in $\mathcal{A}_\delta^h$ or $\mathcal{A}_\delta^{\mm}$. 
\begin{lemma}\label{lambda r nought lemma}
If $\S^\delta$ is a minimizer for $\mathcal{F}$, then there exist $r_0>0$ and $0\leq \Lambda<\infty$, both depending on $\S^\delta$, with the following property:
\begin{enumerate}[label=(\roman*)]
\item if $\delta>0$, $\S^\delta$ minimizes $\mathcal{F}$ among $\mathcal{A}_\delta^h$, $\S'$ satisfies the trace condition \eqref{trace constraint}, and $S_\ell^\delta \Delta S_\ell' \subset B_r(x)$ for $r<r_0$ and $1\leq \ell \leq N$, then, setting $G^\delta=B\setminus \cup_{\ell}S_\ell$ and $G'=B\setminus \cup_{\ell}S_\ell'$,
\begin{align}\label{lambda r nought inequality}
    \mathcal{F}(\S^\delta) \leq \mathcal{F}(\S') + \Lambda\big| |G^\delta| - |G'|\big|\,;
\end{align}
\item if $\delta > 0$, $\S^\delta$ minimizes $\mathcal{F}$ among $\mathcal{A}_\delta^{\mm}$ and $\S'$ satisfies $S_\ell^\delta \Delta S_\ell' \subset B_r(x)$ for $r<r_0$ and $1\leq \ell \leq N$, then
\begin{align}\label{lambda r nought inequality volume version}
    \mathcal{F}(\S^\delta) \leq \mathcal{F}(\S') + \Lambda\sum_{\ell=1}^N \big| |S_\ell^\delta|-|S_\ell'| \big|\,.
\end{align}
\end{enumerate}
\end{lemma}

\begin{proof}
For $(i)$, since we do not have to fix the areas of each chamber but only the remnant set, the proof is an application of the standard volume-fixing variations construction for sets of finite perimeter along the lines of \cite[Lemma 17.21 and Example 21.3]{Mag12}. 
For $(ii)$, we use volume-fixing variations idea for clusters originating in \cite[VI.10-12]{Alm76}. More specifically, by considering the $(N+1)$-cluster $(S_0^\delta,\dots,S_N^\delta,G^\delta)$, \eqref{lambda r nought inequality volume version} follows directly from using \cite[Equations (29.80)-(29.82)]{Mag12} on this $(N+1)$-cluster to modify $\S'$ so that its energy may be tested against $\S^\delta$.\end{proof}

\subsection{Preliminary regularity when $\delta>0$}
Density estimates and regularity along $(G^\delta)^{1/2}\cap (S_\ell^\delta)^{1/2}$ can be derived from Lemma \ref{lambda r nought inequality}.

\begin{lemma}[Infiltration Lemma for $\delta>0$]\label{infiltration lemma}
    If $\S^\delta$ is a minimizer for $\mathcal{F}$ among $\mathcal{A}_\delta^h$ or $\mathcal{A}_\delta^{\mm}$ for $\delta> 0$, then there exist constants $\varepsilon_0=\varepsilon_0>0$ and $r_*>0$ with the following property:
    \par
    \smallskip
    if $x\in \cl B$ when $\S^\delta \in \mathcal{A}_\delta^h$ or $x\in \mathbb{R}^2$ when $\S^\delta \in \mathcal{A}_\delta^{\mm}$, $r<r_*$, $0\leq \ell \leq N$, and
\begin{align}\label{small density assumption lemma}
 |S_\ell^\delta \cap B_r(x)| \leq \varepsilon_0 r^2\,,
\end{align}
then
\begin{align}\label{no infiltration equation}
    |S_\ell^\delta \cap B_{r/4}(x)| =0\,.
\end{align}
\end{lemma}



\begin{proof}We prove the lemma for $\mathcal{A}_\delta^h$ case in steps one and two. The case for $\mathcal{A}_\delta^{\mm}$ is the the same except that one uses \eqref{lambda r nought inequality volume version} instead of \eqref{lambda r nought inequality} when testing minimality in \eqref{on the way to differential inequality} below.
\par
\medskip
\noindent\textit{Step one}: In the first step, we show that there exists $\varepsilon(h)>0$ such that if $x\in \cl B$, $r<1$, and
\begin{align}\label{small density assumption lemma step one}
 |S_\ell^\delta \cap B_r(x)| \leq \varepsilon r^2\quad \textup{ for some $1\leq \ell \leq N$}\,,
\end{align}
for a minimizer among $\mathcal{A}_\delta^h$, then
\begin{align}\label{doesnt affect trace}
    B_{r/2}(x) \cap \{ h=\ell\}=\emptyset\,.
\end{align}
If $B_r(x) \cap \partial B = \emptyset$, \eqref{doesnt affect trace} is immediate, so we may as well assume in addition that 
\begin{align}\label{step one infiltration assumptions}
    B_r(x) \cap \partial B \neq \emptyset \,. 
\end{align}
In order to choose $\varepsilon$, we recall the inclusion \eqref{circular segment containment} from Theorem \ref{existence theorem}, 
which allows us to pick $\varepsilon$ small enough (independent of $\delta$ or the particular minimizer) so that if $y\in \{h=\ell\}$, then
\begin{align}\label{circular segment density assumption}
    \inf_{0<r<1} \frac{|S_\ell^\delta \cap B_r(y)|}{r^2}> 4\varepsilon\,.
\end{align}
Now if $B_r(x)$ satisfies \eqref{small density assumption lemma step one}-\eqref{step one infiltration assumptions}, we claim that
\begin{equation}\label{no S ell trace inside smaller ball}
 B_{r/2}(x)\cap\{h=\ell\} =\emptyset\,,
\end{equation}
which is \eqref{doesnt affect trace}. Indeed, if \eqref{no S ell trace inside smaller ball} did not hold, then we could find $y\in B_{r/2}(x)$ such that $h(y)=\ell$, in which case by \eqref{circular segment density assumption},
\begin{align}\label{too much density inside small ball}
     \frac{|S_\ell^\delta \cap B_{r/2}(y)|}{r^2/4} > 4\varepsilon\,.
\end{align}
But $B_{r/2}(y) \subset B_{r}(x)$, so that \eqref{too much density inside small ball} implies $|S_\ell^\delta \cap B_r(x)| > \varepsilon r^2$, which contradicts our assumption \eqref{small density assumption lemma step one}. 
\par
\medskip
\noindent\textit{Step two}: Let $\varepsilon_0<\varepsilon$ and $r_* < 1$ to be positive constants to be specified later, and suppose that \eqref{small density assumption lemma} holds for some $1\leq \ell\leq N$ and $x\in \cl B$ with $r<r_*$. We set $m(r)=|S_{\ell}^\delta \cap B_r(x)|$, so that for almost every $r$, the coarea formula gives
\begin{align}\label{cons of coarea}
    m'(r) = \mathcal{H}^1((S_{\ell}^\delta)^{(1)} \cap \partial B_r(x)) = \mathcal{H}^1((S_{\ell}^\delta)^{(1)} \cap \partial B_r(x) \cap B)\,.
\end{align}
By the conclusion \eqref{doesnt affect trace} of step one, 
\begin{align}\label{doesnt affect trace step two}
    B_{r/2}(x) \cap \{ h=\ell\}=\emptyset\,.
\end{align}
Therefore, for $s< r/2$,
\begin{align}\label{trace still satisfied}
    (S_{\ell}^\delta \setminus B_s(x))\cap \partial B = \{x\in \partial B : h(x) = \ell \}\quad\textup{in the sense of traces}\,.
\end{align}
In particular, removing $B_s(x)$ from $S_{\ell}^\delta$ does not disturb the trace condition \eqref{trace constraint}. Then we may apply \eqref{lambda r nought inequality} from Lemma \ref{lambda r nought lemma}, yielding for almost every $s<r/2$ 
\begin{align}\notag
    \mathcal{F}(\S^\delta)
    &\leq \mathcal{F}(B^c,S_1^\delta,\dots ,S_{\ell}^\delta \setminus B_s(x), \dots,S_N^\delta, G^\delta \cup (S_\ell^\delta \cap B_s(x))) + \Lambda |S_{\ell}^\delta \cap B_s(x)| \\ \label{on the way to differential inequality}
    &= \mathcal{F}(\S^\delta) - c_\ell P(S_\ell^\delta;B_s(x))+ c_\ell\mathcal{H}^1((S_\ell^\delta)^{(1)} \cap \partial B_s(x))+ \Lambda |S_\ell^\delta \cap B_s(x)|\,;
\end{align}
in the second line we have used the formula
$$
P(S_{\ell}^\delta \setminus B_s(x);B)=P(S_{\ell}^\delta; B\setminus  \cl B_s(x)) + \mathcal{H}^1((S_{\ell}^\delta)^{(1)} \cap \partial B_s(x))\,,
$$
which holds for all but those countably many $s$ with $\mathcal{H}^1(\partial^* S_{\ell}^\delta\cap \partial B_s(x))>0$. After rearranging \eqref{on the way to differential inequality} and using the isoperimetric inequality to obtain
\begin{align}\notag
    2c_{\ell}\pi^{1/2}m(s)^{1/2} \leq  2c_{\ell}m'(s) + \Lambda m(s)\,,
\end{align}
we may reabsorb the term $\Lambda m(s)$ onto the left hand side and integrate to obtain the requisite decay on $m$.
\end{proof}

\begin{corollary}[Regularity along $(G^\delta)^{1/2} \cap  (S_\ell^\delta)^{1/2}$]\label{reduced boundary regularity delta positive}
If $\delta>0$ and, for a minimizer $\S\in \mathcal{A}_\delta^h$ or $\mathcal{S}\in \mathcal{A}_\delta^{\mm}$ and point $x\in B$ or $x\in \mathbb{R}^2$, respectively, there exists $r_j\to 0$ and $ \ell$ such that
\begin{align}\label{sees 1/2}
    1=\lim_{j\to \infty}\frac{|(G^\delta \cup S_\ell^\delta)\cap B_{r_j}(x)|}{\pi r_j^2}\,,
\end{align}
then for large $j$, $\partial G^\delta \cap \partial S_\ell^\delta \cap B_{r_j}(x)$ is an arc of constant curvature and $S_{\ell'} \cap B_{r_j}(x) = \emptyset$ for $\ell'\neq \ell$.
\end{corollary}

\begin{proof}
By our assumption \eqref{sees 1/2} and the infiltration lemma, for some $j$ large enough, $B_{r_j}(x) \subset S_\ell^\delta \cup G^\delta$, in which case the classical regularity theory for volume-constrained minimizers of perimeter gives the conclusion.
\end{proof}

\begin{corollary}[Density Estimates]\label{density corollary}
    If $\S^\delta$ minimizes $\mathcal{F}$ among $\mathcal{A}_\delta^h$ or $\mathcal{A}_\delta^{\mm}$ for some $\delta> 0$, then there exists $0<\alpha_1,\, \alpha_2< 1$ and $r_{**}>0$  such that if $x\in \partial S_\ell^\delta$, then for all $r<r_{**}$,
\begin{align}\label{area density equation}
    \alpha_1 \pi r^2 \leq |S_\ell^\delta \cap B_r(x)| &\leq (1-\alpha_1)\pi r^2 \\ \label{perimeter density equation}
    P(S_\ell^\delta;B_r(x)) &\leq \alpha_2r\,.
\end{align}
Also, $\mathcal{H}^1(\partial S_\ell^\delta\setminus \partial^* S_\ell^\delta)=0$ and each $(S_\ell^\delta)^{(1)}$ and $(G^{\delta})^{(1)}$ is open and satisfies \eqref{boundary convention 1}-\eqref{boundary convention 2}.
\end{corollary}

\begin{proof}
    We consider the case $\S^\delta \in \mathcal{A}_\delta^h$ and $1\leq \ell \leq N$, and the other cases are similar. First we prove the lower bound in \eqref{area density equation}. Let $x\in \partial S_\ell^\delta$. Then by our convention \eqref{boundary convention 1}-\eqref{boundary convention 2} regarding topological boundaries,
\begin{align}\notag
    |S_\ell^\delta \cap B_r(x)| > 0 \quad\textup{for all $r>0$}\,.
\end{align}
  By the infiltration lemma, the lower area density bound follows with $\alpha_1=\varepsilon_0$ and $r_{**}=r_*$. 
    \par
For the upper area bound, let us choose $r_{**}\leq r_*$ such that
\begin{align}\label{scale invariant small}
    \Lambda r_{**} \leq 1\,.
\end{align}
We claim that for any $x\in \partial S_\ell^\delta$ and $r<r_{**}$, 
\begin{align}\label{upper area bound when delta positive}
    |S_\ell^\delta \cap B_r(x)| \leq \max\left\{\pi - \varepsilon_0,c_* \right\} r^2
\end{align}
for a dimensional constant $c_*$ to be specified shortly. Suppose that this were not the case. Then by the smoothness of $\partial B$ and the containment of $S_\ell^\delta$ in $B$, 
\begin{align}\label{far from boundary 2}
    \dist (x, \partial B ) \geq c(B) r
\end{align}
for some constant $c(B)<1/2$ depending $B$, so that $B_{c(B)r/2}(x)\subset B$. Also, by the infiltration lemma, $S_{\ell'}^\delta \cap B_{r/4}(x)=\emptyset$ for $\ell' \neq \ell$. 
These two facts combined imply that $B_{c(B)r/2}(x) \subset\!\subset S_\ell^\delta \cup G^\delta$. By Lemma \ref{lambda r nought lemma}, $S_\ell^\delta$ is a $(\Lambda, r_{**})$-minimizer of perimeter in $B_{c(B)r/2}(x)$ with $\Lambda r_{**}\leq 1$ by \eqref{scale invariant small}. Then the density estimates \cite[Theorem 21.11]{Mag12} for these minimizers give 
\begin{align}
    |S_\ell^\delta \cap B_{c(B)r/2}(x)| \leq \frac{15\pi}{64}(c(B)r)^2\,.
\end{align}
By choosing $c_{**}$ close enough to $\pi$, we have a contradiction.The upper area bound follows from a construction which we omit, and the mild regularity on $\partial S_\ell^\delta$ follows from our normalization \eqref{boundary convention 1}-\eqref{boundary convention 2}, the area bounds, and Federer's theorem \cite[4.5.11]{Fed69}.\end{proof}

\begin{remark}[Lebesgue representatives]\label{lebesgue representative remark}
   For the rest of the paper, we will always assume that we are considering the open set $(S_\ell^\delta)^{(1)}$ or $(G^\delta)^{(1)}$ as the Lebesgue representative of $S_\ell^\delta$ or $G^\delta$. 
\end{remark}

\subsection{Preliminary regularity when $\delta=0$}

The following infiltration (or ``elimination") lemma for a minimizer among $\mathcal{A}_0^{\mm}$ is due to \cite[Theorem 3.1]{Leo01} and can be adapted easily to the problem on the ball; the reader may also consult \cite[Section 11]{White} for a similar statement.

\begin{lemma}[Infiltration Lemma for $\delta=0$]\label{infiltration lemma 2}
    If $\S^0$ is a minimizer for $\mathcal{F}$ among $\mathcal{A}_0^h$, then there exist constants $\varepsilon_0=\varepsilon_0>0$ and $r_*>0$ with the following property:
    \par
    \smallskip
    if $x\in \mathbb{R}^2$, $r<r_*$, and $0\leq \ell_0< \ell_1\leq N$ are such that
\begin{align}\label{small density assumption lemma 2}
 |B_r(x) \setminus (S_{\ell_0}^0\cup S_{\ell_1}^0)| \leq \varepsilon_0 r^2\,,
\end{align}
then
\begin{align}\label{no infiltration equation 2}
    |B_{r/4}(x) \setminus (S_{\ell_0}^0\cup S_{\ell_1}^0)| =0\,.
\end{align}
\end{lemma}

\begin{proof}
Repeating the argument from step one of Lemma \ref{infiltration lemma}, there exists $\varepsilon(h)>0$ such that if $x\in \cl B$, $r<1$, and
\begin{align}\label{small density assumption lemma step one 2}
 |S_\ell^\delta \cap B_r(x)| \leq \varepsilon r^2\quad \textup{ for some $\ell\in\{0,\dots, N\}\setminus \{\ell_0,\ell_1\}$}\,,
\end{align}
for a minimizer among $\mathcal{A}_\delta^h$, then
\begin{align}\label{doesnt affect trace 2}
    B_{r/2}(x) \cap \{ h=\ell\}=\emptyset\,.
\end{align}
In particular, by using Lemma \ref{lambda r nought lemma}, we may compare the minimality of $\S^0$ against competitors constructed by donating $B_s(x)\setminus (S_{\ell_0}^0 \cup S_{\ell_1}^0)$ to $S_{\ell_0}$ or $S_{\ell_1}$. The remainder of the argument is the same as in \cite{Leo01}.
\end{proof}

The next two results may be proved as in Corollary \ref{reduced boundary regularity delta positive} and Corollary \ref{density corollary}.

\begin{corollary}[Regularity along $(S_\ell^0)^{1/2} \cap (S_{\ell'}^0)^{1/2}$]\label{reduced boundary regularity delta 0}
If $\mathcal{S}^0$ is a minimizer among $\mathcal{A}_0^h$ and $x\in (S_\ell^0)^{1/2} \cap (S_{\ell'}^0)^{1/2}$ for $\ell,\ell'\in \{1,\dots,N\}$, then in a neighborhood of $x$, every other chamber is empty and $\partial S_\ell^0 \cap \partial S_{\ell'}^0$ is a segment.
\end{corollary}

\begin{lemma}[Upper Area and Perimeter Bounds]\label{upper perimeter bounds delta 0}
     If $\S^0$ minimizes $\mathcal{F}$ among $\mathcal{A}_0^h$, then there exists $\alpha_3>0$, $\alpha_4<1$, and $r_3>0$,  such that 
\begin{align} \label{perimeter density equation 2}
    \mathcal{F}(\S^0;B_r(x)) &\leq \alpha_3r\quad \forall\, r>0\,,\quad   x\in \mathbb{R}^2\,,
\end{align}
and
\begin{align}\label{upper area bound delta zero}
    |B_r(x) \cap S_\ell^0| \leq \alpha_4 \pi r^2 \quad \forall x\in \partial S_{\ell}^0\,,\quad r<r_3\,.
\end{align}
\end{lemma}

\subsection{Monotonicity formula} This is the last technical tool necessary for obtaining blow-up cones.
\begin{theorem}[Monotonicity Formula]\label{mon formula}
   If $\S^\delta$ minimizes $\mathcal{F}$ among $\mathcal{A}_\delta^h$ for $\delta\geq 0$ or $\mathcal{A}_\delta^{\mm}$ for $\delta> 0$, then there exists $\Lambda_0\geq 0$ such that if $x\in \mathbb{R}^2$,
\begin{align}\label{mon equation}
  \sum_{\ell=0}^N\frac{c_\ell}{2}
  &\int_{\partial^* S_\ell^\delta \cap (B_r(x) \setminus B_s(x))}\frac{((y-x)\cdot \nu_{S_\ell^\delta})^2}{|y-x|^3} \,d\mathcal{H}^1(y)\leq  \frac{\mathcal{F}(\S^\delta;B_r(x))}{r} - \frac{\mathcal{F}(\S^\delta;B_s(x))}{s} +\Lambda_0 (r-s)
\end{align}
for any $0<s<r<r_x$.
\end{theorem}

\begin{proof}
We consider the case $\S^\delta$ is minimal among $\mathcal{A}_\delta^h$ and $x\in \partial B$ is a jump point of $h$; the other cases are simpler since the trace constraint may be avoided. First, we observe that by the smoothness of the circle, there exists $\Lambda'>0$ such that
\begin{align}\notag
    \sum_{\ell=0}^N\frac{c_\ell}{2}
  \int_{\partial^* S_\ell^\delta \cap (B_r(x) \setminus B_s(x)\cap \partial B)}\frac{((y-x)\cdot \nu_{S_\ell^\delta})^2}{|y-x|^3} \,d\mathcal{H}^1(y)&\leq  \frac{\mathcal{F}(\S^\delta;B_r(x) \cap \partial B)}{r} - \frac{\mathcal{F}(\S^\delta;B_s(x)\cap \partial B)}{s} \\ \notag
  &\qquad +\Lambda'\pi (r-s) \quad \forall 0<s<r<r_x
\end{align}
for some $r_x$; that is, we have the desired monotonicity for energy along $\partial B$. For the remainder of the proof, we therefore focus on the energy inside $B$. We define the increasing function 
\begin{align}\label{p definition}
    p(r) := \sum_{\ell=1}^N c_\ell P(S_\ell^\delta ; B_r(x)\cap B) = \mathcal{F}(\S;B_r(x) \cap B)
\end{align}
where, since it will be clear by context, we have suppressed the dependence of $p$ on $x$. The proof requires two steps: first, deriving a differential inequality for $p$ using comparison with cones (see \eqref{testing against cones}), and second, integrating and employing a slicing argument. The computations in the second step are the same as those in the proof of the monotonicity formula for almost minimizing integer rectifiable currents \cite[Proposition 2.1]{DeLSpaSpo17}, so we omit them. 
\par
We prove that given $x\in \partial B$ which is a jump point of $h$, there exists $r_{x}>0$ such that
\begin{align}\label{testing against cones}
   \frac{p(r)}{r^2} \leq  \frac{1}{r}\sum_{\ell=1}^N c_\ell \mathcal{H}^0(\partial^* S_\ell^\delta \cap \partial B_r(x) \cap B)+ \Lambda \pi \quad\textup{ for a.e. }r<r_x\,,
\end{align}
where $\Lambda$ is from Lemma \ref{lambda r nought lemma}. As mentioned above, the monotonicity formula can then be derived from \eqref{testing against cones}. 
For concreteness, suppose that $h$ jumps between $1$ and $2$ at $x$. Then, recalling \eqref{chord def} there are chords $c_i^1$ and $c_j^2$ connecting $x$ to the nearest jump points on either side and corresponding circular segments $C_i^1$ and $C_j^2$. Let $0<r_x<r_0$ be small enough such that $\cl B_{r_x}(x)$ intersects no other circular segments from \eqref{circular segments} other  than those two. 
By the inclusion \eqref{circular segment containment} for the minimizer and our choice of $r_x$, for every $r<r_x$, 
\begin{align}\label{circular segment containment on small circle}
    C_i^1 \cap \cl B_r(x) \subset S_1^\delta\,,\quad  C_j^2 \cap \cl B_r(x) \subset S_2^\delta\,,\quad\textup{and}\quad \partial B_r(x) \cap \partial B \subset \partial C_i^1 \cup \partial C_j^2 \,.
\end{align}
For $r<r_x$ to be specified momentarily, we consider the cluster $\tilde{\S}$ defined by
\begin{align}\notag
    \tilde{S}_1 &= (S_1^\delta \setminus \cl B_r(x)) \cup \{y\in B_r(x) \setminus C_i^1 :  x + r(y-x)/|y-x| \in S_1^\delta\} \cup C_i^1\,, \\ \notag
    \tilde{S}_2 &= (S_2^\delta \setminus \cl B_r(x)) \cup \{y\in B_r(x) \setminus C_j^2 :  x + r(y-x)/|y-x| \in S_2^\delta\} \cup C_j^2\,, \\ \notag
    \tilde{S}_\ell &= (S_\ell^\delta \setminus \cl B_r(x)) \cup \{y\in B_r(x) :  x + r(y-x)/|y-x| \in S_\ell^\delta\}\,,\quad 3\leq \ell \leq N\,. 
\end{align}
Note that by \eqref{circular segment containment on small circle}, each $\partial \tilde{S}_\ell \cap B_r(x)$ consists of radii of $B_r(x)$ contained in $B_r(x) \setminus (C_i^1 \cup C_j^2)$. Then by Lemma \ref{slicing lemma}, for almost every $r<r_x$, each $\tilde{S}_\ell$ is a set of finite perimeter and
\begin{align}\notag
    \sum_{\ell=1}^N c_\ell P(\tilde{S}_\ell; B )&= \sum_{\ell=1}^N c_\ell P(\tilde{S}_\ell; B_r(x) \cap B) + \sum_{\ell=1}^N c_\ell P(\tilde{S}_\ell; B \setminus \cl B_r(x)) \\ \label{cone perimeter}
    &= r \sum_{\ell=1}^N c_\ell \mathcal{H}^0(\partial^* S_\ell^\delta \cap \partial B_r(x) \cap B) + \sum_{\ell=1}^N c_\ell P(S_\ell^\delta; B \setminus \cl B_r(x))\,.
\end{align}
Also, by \eqref{circular segment containment on small circle} and our definition of the $\tilde{\S}$, $\tilde{\S}$ satisfies the trace condition \eqref{trace constraint}. If we set $\tilde{G} = B \setminus (\cup_\ell \tilde{S}_\ell)$, then we can plug \eqref{cone perimeter} into the comparison inequality \eqref{lambda r nought inequality} from Lemma \ref{lambda r nought lemma} and cancel like terms, yielding  
\begin{align} \notag
    p(r)=\sum_{\ell=1}^N c_\ell P(S_\ell^\delta; B \cap B_r(x)) 
    \leq r \sum_{\ell=1}^N c_\ell \mathcal{H}^0(\partial^* S_\ell^\delta \cap \partial B_r(x) \cap B) + \Lambda \pi r^2 \quad\textup{for a.e. $r<r_x$}\,.
\end{align} 
This is precisely \eqref{testing against cones} multiplied by $r^2$.
\end{proof}

\subsection{Existence of blow-up cones}

The monotonicity formula allows us to identify blow-up minimal cones at interfacial points of a minimizer. It will be convenient to identify interfacial points for minimizers among $\mathcal{A}_\delta^{\mm}$ with interfacial points in $B$ for minimizers among $\mathcal{A}_\delta^h$, since, at the level of  blow-ups, the behavior is the same.

\begin{definition}[Interior and boundary interface points]\label{int bound int pt def}
    If $\mathcal{S}^\delta$ is minimal among $\mathcal{A}_\delta^h$ and $x\in B \cap \partial S_\ell^\delta$ or $\mathcal{S}^\delta$ is minimal among $\mathcal{A}_\delta^{\mm}$ and $x\in \partial S_\ell^\delta$ for some $\ell$, we say $x$ is an {\bf interior interface point}. If $\mathcal{S}^\delta$ is minimal among $\mathcal{A}_\delta^h$ and $x\in \partial B$, we call $x$ a {\bf boundary interface point}.
\end{definition}

The blow-ups at a boundary interface point will be minimal in a halfspace among competitors satisfying a constraint coming from the trace condition \eqref{trace constraint} and the inclusion \eqref{circular segment containment} from Theorem \ref{existence theorem}.

\begin{definition}[Admissible blow-ups at jump points of $h$]\label{boundary blow up definition}
Let $x\in \partial B$ be a jump point of $h$, and let $C_i^{\ell_0}$ and $C_j^{\ell_1}$ be the circular segments from \eqref{circular segments} meeting at $x$. Let
\begin{align}\label{circular segment blowup}
    C_\infty^{\ell_0}=\bigcup_{\lambda >0} \lambda (C_i^{\ell_0}-x)\,,\quad  C_\infty^{\ell_1}=\bigcup_{\lambda >0} \lambda (C_j^{\ell_1}-x)
\end{align}
be the blow-ups of the convex sets $C_i^{\ell_0}$ and $C_j^{\ell_1}$ at their common boundary point $x$. We define
\begin{align}\label{ax definition}
    \mathcal{A}_x:=\{\S : \forall \ell \neq 0,\,S_\ell \subset \{y \cdot x < 0 \}=S_0^c\,,\,\,S_\ell \cap S_{\ell'}=\emptyset \textup{ if }\ell\neq \ell'\,,\,\,C_\infty^{\ell_k}\subset S_{\ell_k}\textup{ for }k=0,1\}\,.
\end{align}
\end{definition}

\begin{theorem}[Existence of Blow-up Cones]\label{existence of cones}
If $\S^\delta$ minimizes $\mathcal{F}$ among $\mathcal{A}_\delta^h$ for some $\delta \geq 0$ or among $\mathcal{A}_\delta^{\mm}$ for $\delta>0$, then for any sequence $r_j \to 0$, there exists a subsequence $r_{j_k}\to 0$ and cluster $\S=(S_0,\dots,S_N,G)$ partitioning $\mathbb{R}^2$, satisfying the following properties: 
\begin{enumerate}[label=(\roman*)]
\item $(S_\ell^\delta - x)/r_{j_k} \overset{L^1_\loc}{\to} 
 S_\ell$ for each $0\leq \ell\leq N$;
 \item $\mathcal{H}^1\mres (\partial S_\ell^\delta - x)/r_{j_k}\weakstar \mathcal{H}^1 \mres \partial S_\ell$ for each $0\leq \ell\leq N$;
\item $S_\ell$ is an open cone for each $0\leq \ell \leq N$;
\item if $x$ is an interior interface point and $\tilde{\S}$ is such that for $0\leq \ell \leq N$, $\tilde{S}_\ell\Delta  S_\ell \subset\!\subset B_R$ and, for the problem on the ball, $S_0=\emptyset$, then 
\begin{align}\label{interior cone is minimal}
\mathcal{F}(\S; B_R) \leq \mathcal{F}(\tilde{\S}; B_R);
\end{align}
\item if $x\in \partial S_{\ell_0}^\delta \cap \partial B$ is not a jump point of $h$, then $S_{\ell_0} = \{y: y\cdot x < 0 \}$ and $S_0=\{y:y\cdot x > 0\}$;
\item if $x\in \partial S_{\ell_0}^\delta \cap \partial B$ is a jump point of $h$, then $\S \in \mathcal{A}_x$ 
and if $\tilde{\S}\in \mathcal{A}_x$ is such that for $0\leq \ell \leq N$, $\tilde{S}_\ell\Delta  S_\ell \subset\!\subset B_R$, then 
\begin{align}\label{boundary cone is minimal}
\mathcal{F}(\S; B_R ) \leq \mathcal{F}(\tilde{\S}; B_R )\,.
\end{align}
\end{enumerate}
\end{theorem}

\begin{proof}
    When $x$ is a boundary interface point and is not a jump point of $h$, then $S_{\ell_0}^\delta \cap B \cap B_{r_x}(x)= B \cap B_{r_x}(x)$ for some $r_x>0$ by \eqref{circular segment containment} from Theorem \ref{existence theorem}. In this case, items $(i)$-$(iii)$ and $(v)$ are trivial. Also, the case of interior interface points is essentially a simpler version of the argument when $x\in \partial S_{\ell_0}^\delta  \cap \partial B$ is a jump point of $h$. Therefore, for the rest of the proof, we focus on items $(i)$-$(iii)$ and $(vi)$ when $x\in \partial B \cap \partial S_{\ell_0}^\delta$ is a jump point of $h$.
    \par
The upper perimeter bounds from Corollary \ref{density corollary} or Lemma \ref{upper perimeter bounds delta 0} and compactness in $BV$ give the existence of $r_{j_k}\to 0$ and $\S$ such that the convergence in $(i)$ holds. In addition, this compactness gives
\begin{align}\label{convergence of vector measures}
   \mu_k^\ell:= (\nu_{ S_\ell^\delta } \mathcal{H}^{1} \mres \partial S_\ell^\delta - x)/r_{j_k}\weakstar \nu_{S_\ell }\mathcal{H}^1 \mres \partial^* S_\ell=: \mu_\ell \quad \forall 0\leq \ell \leq N\,.
\end{align}
It is easy to check from the inclusion \eqref{circular segment containment} that $\S \in \mathcal{A}_x$. We now discuss in order \eqref{boundary cone is minimal}, $(ii)$, and $(iii)$. The proofs of \eqref{boundary cone is minimal} and $(ii)$ are standard compactness arguments that proceed mutatis mutandis as the proof of the compactness theorem for $(\Lambda,r_0)$-perimeter minimizers \cite[Theorem 21.14]{Mag12}. Finally, $(iii)$ follows from the monotonicity formula \eqref{mon equation}, which implies that $\mathcal{F}(\S;B_r)/r$ is constant in $r$, and the characterization of cones \cite[Proposition 28.8]{Mag12}. 
\end{proof}

\subsection{Classification of blow-up cones}\label{sec: classification of blowup cones}

We classify the possible blow-up cones for a minimizer using the terminology set forth in Definition \ref{int bound int pt def}. 

\begin{theorem}[Classification of Blow-up Cones for $\delta>0$]\label{positive delta classification theorem}
    If $\S^\delta$ minimizes $\mathcal{F}$ among $\mathcal{A}_\delta^h$ or $\mathcal{A}_\delta^{\mm}$ for some $\delta > 0$, and $\S$ is a blow-up cluster for $x\in \partial S_{\ell_0}^\delta$ and some $r_{j}\to 0$, then exactly one of the following is true:
\begin{enumerate}[label=(\roman*)]
\item $x\in \partial S_{\ell_0}^\delta$ is an interior interface point and $S_{\ell_0}=\{y\cdot \nu_{S_{\ell_0}^\delta}(x)<0\}$, $G = \mathbb{R}^2 \setminus S_{\ell_0}$;
\item $x\in \partial S_{\ell_0}^\delta$ is an interior interface point, $S_{\ell_0}=\{y\cdot \nu <0\}$ for some $\nu \in \mathbb{S}^1$, and $S_{\ell_1}=\mathbb{R}^2 \setminus S_{\ell_0}$ for some $\ell_1\neq \ell_0$;
\item $x\in \partial S_{\ell_0}^\delta \cap \partial B$ is a boundary interface point and jump point of $h$, $S_{\ell_0}=\{y\cdot \nu <0,\, y \cdot x < 0\}$, and $S_{\ell_1}=\{y\cdot \nu >0,\, y \cdot x < 0\}$ for some $\nu \in \mathbb{S}^1$ and $\ell_1 \neq \ell_0$;
\item $x\in \partial S_{\ell_0}^\delta \cap \partial B$ is a boundary interface point and jump point of $h$, $S_{\ell_0}=\{y\cdot \nu_0 <0,\, y \cdot x < 0\}$, $S_{\ell_1}=\{y\cdot \nu_1 >0,\, y \cdot x < 0\}$, $S_0=\{y\cdot x >0\}$, and $G= (S_0 \cup S_{\ell_0}\cup S_{\ell_1})^c$ for some $\nu_0,\, \nu_1 \in \mathbb{S}^1$ and $\ell_1 \neq \ell_0$;
\item $x\in \partial S_{\ell_0}^\delta \cap \partial B$ is a boundary interface point, not a jump point of $h$, $S_{\ell_0}=\{y\cdot x <0\}=S_0^c$.
\end{enumerate}
\end{theorem}


\begin{proof}[Proof of Theorem \ref{positive delta classification theorem}]
\textit{Step one}: In this step we consider an interior interface point $x$ and show that either $(i)$ or $(ii)$ holds. First, we note that since $x\in \partial S_{\ell_0}^\delta$ and the density estimates \eqref{area density equation} pass to the blow-up limit, $S_{\ell_0}\neq \emptyset$ and $S_{\ell_0} \neq\mathbb{R}^2$, so $\S$ is non-trivial. We claim that no non-empty connected component $S_\ell$ of $\S$ for $0\leq \ell \leq N$ can be anything other than a halfspace; from this claim it follows that $(i)$ or $(ii)$ holds. Indeed, suppose that there was such a component $C$, say of $S_1$, defined by an angle $\theta\neq \pi$ with $\partial C \cap \partial B=\{c_1,c_2\}$. Let $K$ be the convex hull of $c_1$, $c_2$, and $0$. If $\theta<\pi$, then the triangle inequality implies that the cluster $\mathcal{S}'=(S_0,S_1 \setminus K, S_2,\dots, S_N, G\cup K)$ satisfies $\mathcal{F}(\S';B_2)<\mathcal{F}(\S;B_2)$, contradicting the minimality property \eqref{interior cone is minimal}. On the other hand, if $\theta>\pi$, then the cluster $\S'=(S_0\setminus K,S_1 \cup K, S_2 \setminus K, \dots, S_N \setminus K, G\setminus K)$ also contradicts \eqref{interior cone is minimal} due to the triangle inequality. 
\par
\medskip
\noindent\textit{Step two}: Moving on to the case of a boundary interface point, we begin by observing that $(v)$ is trivial by \eqref{circular segment containment} when $x$ is not a jump point of $h$. If $x$ is a jump point of $h$, say between $h=1$ and $h=2$, then $S_0=\{y\cdot x > 0\}$, and $\{S_1,\dots,S_N,G\}$ partition $S_0^c$. Now the same argument as in the previous step using the triangle inequality shows that $S_\ell=\emptyset$ for $3 \leq \ell \leq N$ and $S_1$ and $S_2$ each only have one connected component bordering $S_0$. It follows that either $(iii)$ or $(iv)$ holds.\end{proof}

\begin{corollary}[Regularity for $\partial G^{\delta}$ away from $(G^\delta)^{(0)}$]\label{no termination corollary}
If $\S^\delta$ minimizes $\mathcal{F}$ among $\mathcal{A}_\delta^h$ or $\mathcal{A}_\delta^{\mm}$ for some $\delta > 0$ and $x$ is an interior interface point such that
\begin{align}\label{limsup assumption}
    \limsup_{r\to 0} \frac{|G^{\delta} \cap B_r(x)|}{\pi r^2} > 0\,,
\end{align}
then there exists $r_x>0$ such that $\partial G^{\delta}\cap B_{r_x}(x)$ is an arc of constant curvature dividing $B_{r_x}(x)$ into $G^{\delta} \cap B_{r_x}(x)$ and $S_{\ell}^\delta \cap B_{r_x}(x)$.
\end{corollary}

\begin{proof}
    If $r_j\to 0$ is a sequence achieving the limit superior in \eqref{limsup assumption}, then any subsequential blow-up at $x$ must be characterized by case $(i)$ of Theorem \ref{positive delta classification theorem}. The desired conclusion now follows from Corollary \ref{reduced boundary regularity delta positive}. 
\end{proof}

Lastly, we classify blow-up cones for $\delta=0$ when either $N=3$ or the weights are equal.

\begin{theorem}[Classification of Blow-up Cones for $\delta=0$ on the Ball]\label{zero delta classification theorem}
    If $N=3$ or $c_\ell=1$ for $0\leq \ell \leq N$, $\S^0$ minimizes $\mathcal{F}$ among $\mathcal{A}_0^h$, and $\S$ is a blow-up cluster at an interface point $x$, then exactly one of the following is true:
\begin{enumerate}[label=(\roman*)]
\item $x\in \partial^* S_{\ell_0}^0 \cap \partial^* S_{\ell_1}^0$ is an interior interface point and $S_{\ell_0}=\{y:y\cdot \nu_{S_{\ell_0}^\delta}(x)<0\}$, $S_{\ell_1} = \mathbb{R}^2 \setminus S_{\ell_0}$;
\item $x$ is an interior interface point, and the non-empty chambers of $\S$ are three connected cones $S_{\ell_i}$, $i=0,1,2$, with vertex at the origin satisfying 
\begin{align}\label{classical angle conditions}
    \frac{\sin \theta_{\ell_0}}{c_{\ell_1}+c_{\ell_2}}=\frac{\sin \theta_{\ell_1}}{c_{\ell_0}+c_{\ell_2}}=\frac{\sin \theta_{\ell_2}}{c_{\ell_0}+c_{\ell_1}}
\end{align}
where $\theta_{\ell_i}=\mathcal{H}^1(S_{\ell_i} \cap \partial B)$;
\item $x\in \partial S_{\ell_0}^\delta \cap \partial B$ is not a jump point of $h$, and $S_{\ell_0} = \{y: y\cdot x <0 \}=S_0^c$;
\item $x\in \partial S_{\ell_0}^\delta \cap \partial B$ is a jump point of $h$, $S_{\ell_0}=\{y:y\cdot \nu <0,\, y \cdot x < 0\}$, and $S_{\ell_1}=\{y:y\cdot \nu >0,\, y \cdot x < 0\}$ for some $\nu \in \mathbb{S}^1$ and $\ell_1 \neq \ell_0$, and $S_0=\{y:y\cdot x >0\}$;
\item $x\in \partial S_{\ell_0}^\delta \cap \partial B$ is a jump point of $h$, and the non-empty chambers of $\S\in\mathcal{A}_x$ are $S_0=\{y:y\cdot x >0\}$ and three connected cones $S_{\ell_i}$, $i=0,1,2$, partitioning $S_0^c$. 
\end{enumerate}
\end{theorem}

\begin{proof} We begin with the observation that no blow-up at $x$ can consist of a single chamber. To see this, since $x$ is an interface point, it belongs to $\partial S_{\ell}^0$ for some $\ell$. By our normalization \eqref{boundary convention 1}-\eqref{boundary convention 2} for reduced and topological boundaries, $x\in \spt \,\mathcal{H}^1\mres \partial^* S_{\ell}^0$. Therefore, due to the upper area bound \eqref{upper area bound delta zero}, no blow-up limit at $x$ can consist of a single chamber $S_{\ell'}$; if so, the $L^1$ convergence and the infiltration lemma would imply that $x\in \mathrm{int}\, S_{\ell'}^0$, contradicting $x\in \spt \,\mathcal{H}^1 \mres \partial^* S_{\ell}^0$. Therefore, there are at least two chambers in the blow-up cluster at $x$.
\par
Next, we claim that when $N=3$ or $c_\ell=1$ for all $\ell$, there cannot be four or more non-empty connected components of chambers of $\S$ comprising $\mathbb{R}^2$ if the blow-up is at an interior interface point or comprising $\{y:y\cdot x<0\}$ at a boundary interface point. If $N=3$ and this were the case, then there must be some $S_\ell$, say $S_1$, which has two connected components $C_1$ and $C_2$ separated by a circular sector $C_3$ with $\partial C_3 \cap \partial B =\{c,c'\}$ and $\dist(c,c')<2$. We set $K$ to be the convex hull of $0$, $c$, and $c'$ and define the cluster $\S'=(S_0, S_1 \cup K, S_2\setminus K, S_3\setminus K,\emptyset)$. Note $\S'\in \mathcal{A}_x$ when $x$ is a boundary interface point. Then the triangle inequality implies that $\mathcal{F}(\S';B_2)<\mathcal{F}(\S;B_2)$, which contradicts the minimality condition \eqref{interior cone is minimal} or \eqref{boundary cone is minimal}. For the case when $c_\ell=1$ for all $\ell$, if there was more than three connected components, there must be some component $C\subset S_\ell$ with $\mathcal{H}^1(C \cap B_1)<2\pi/3$, and when $x$ is a boundary interface point, $\partial C \cap \{y\cdot x=0\}=\{0\}$. Then the construction in \cite[Proposition 30.9]{Mag12}, in which triangular portions of $C$ near $0$ are allotted to the neighboring chambers allows us to construct a competitor (belonging to $\mathcal{A}_x$ if required) that contradicts the minimality \eqref{interior cone is minimal} or \eqref{boundary cone is minimal}.
\par
We may now conclude the proof. If $x$ is an interior interface point, then there are either two or three distinct connected chambers in the blow-up at $x$. Similar to the previous theorem, the triangle inequality implies that if there are two, they are both halfspaces. If there are three, the angle conditions \eqref{classical angle conditions} follow from a first variation argument. If $x$ is a boundary interface point, then $(iii)$ holds by \eqref{circular segment containment} if $x$ is not a jump point of $h$. If $x$ is a jump point of $h$, then $\{y\cdot x < 0\}$ is partitioned into either two or three connected cones. The former is case $(iv)$, and the latter is case $(v)$. 
\end{proof}

\section{Proof of Theorem \ref{main regularity theorem delta positive}}\label{sec: proofs of main regularity theorems}

To streamline the statement below, the terminology ``arc of constant curvature" includes segments in addition to circle arcs. 

\begin{theorem}[Interior Resolution for $\delta>0$]\label{interior resolution theorem delta positive}
    If $\S^\delta$ minimizes $\mathcal{F}$ among $\mathcal{A}_\delta^h$ or $\mathcal{A}_\delta^{\mm}$ for some $\delta>0$ and $x\in \partial S_{\ell_0}^\delta$ is an interior interface point, then there exists $r_x>0$ such that exactly one of the following is true:
\begin{enumerate}[label=(\roman*)]
\item $S_{\ell'}^\delta \cap B_{r_x}(x)=\emptyset$ for $\ell' \neq \ell_0$ and $\partial S_{\ell_0}^\delta \cap B_{r_x}(x)$ is an arc of constant curvature separating $S_{\ell_0}^\delta \cap B_{r_x}(x)$ and $G^\delta \cap B_{r_x}(x)$;
\item $\partial S_{\ell_0}^\delta \cap B_{r_x}(x)$ is an arc of constant curvature separating $B_{r_x}(x)$ into $S_{\ell_0}^\delta \cap B_{r_x}(x)$ and $S_{\ell'}^\delta \cap B_{r_x}(x)$ for some $\ell'\neq \ell_0$;
\item there exist circle arcs $a_1$ and $a_2$ meeting tangentially at $x$ such that 
\begin{align}\notag
    \partial S_{\ell_0}^\delta\cap \partial G^\delta \cap B_{r_x}(x)= a_1\,,\quad \partial S_{\ell'}^\delta\cap \partial G^\delta \cap B_{r_x}(x)= a_2\,, \quad \partial S_{\ell_0}^\delta \cap \partial S_{\ell'}^\delta \cap B_{r_x}(x)=\{x\}\,;
\end{align}
\item there exists circle arcs $a_1$ and $a_2$ meeting in a cusp at $x$ and an arc $a_3$ of constant curvature reaching the cusp tangentially at $x$, and
\begin{align}\notag
    \partial S_{\ell_0}^\delta\cap \partial G^\delta \cap B_{r_x}(x)= a_1\,,\quad \partial S_{\ell'}^\delta\cap \partial G^\delta \cap B_{r_x}(x)= a_2\,, \quad \partial S_{\ell_0}^\delta \cap \partial S_{\ell'}^\delta \cap B_{r_x}(x)=a_3\,.
\end{align}
\end{enumerate}
\end{theorem}

\begin{proof} Let us assume for simplicity that $x$ is the origin; the proof at any other point is similar.
\par
\medskip
\noindent\textit{Step zero}: If $0\notin \partial S_{\ell'}^\delta$ for all $\ell'\neq \ell_0$, then by the density estimates \eqref{area density equation}, $B_{r_0}\cap S_{\ell'}^\delta =\emptyset$ for some $r_0$ and all $\ell'\neq \ell_0$. From the classification of blowups in Theorem \ref{positive delta classification theorem}, $(i)$ must hold at $0$.
\par
\medskip
\noindent\textit{Step one}: For the rest of the proof, we assume instead that for some $\ell'\neq \ell_0$, $0\in \partial S_{\ell'}^\delta$. By Theorem \ref{positive delta classification theorem} and the fact that the density estimates \eqref{area density equation} pass to all blow-up limits, we are in case $(ii)$ of that theorem: any possible blow-up limit at $0$ is a pair of halfspaces coming from $S_{\ell_0}^\delta$ and $S_{\ell'}^\delta$. In this step we identify a rectangle $Q'$ small enough such that $S_{\ell_0}^\delta \cap Q'$ and $S_{\ell'}^\delta \cap Q'$ are a hypograph and epigraph, respectively, over a common axis. 
\par
Let us fix $r_j\to 0$ such that applying Theorem \ref{positive delta classification theorem} and rotating if necessary, we obtain
\begin{align}\label{L1 convergence blowup}
    S_{\ell_0}^\delta/r_j \overset{L^1_\loc}{\to} \mathbb{H}^-:= \{y: y\cdot e_2 < 0 \}\,,&\quad S_{\ell'}^\delta/r_j \overset{L^1_\loc}{\to} \mathbb{H}^+:= \{y: y\cdot e_2 > 0 \}\,, \\ \label{blow-up weak star}
    \mathcal{H}^1 \mres \partial S_{\ell_0}/r_j\,, \,\mathcal{H}^1 \mres \partial &S_{\ell'}/r_j \weakstar\mathcal{H}^1 \mres \partial \mathbb{H}^+\,.
\end{align}
Set
\begin{align}\notag
    Q = [-1,1]\times [-1,1]\,.
\end{align}
We note that for all $r<r_{**}/r_j$,
\begin{align}\label{area density equation blowup}
    \alpha_1 \pi r^2 \leq |S_\ell^\delta/r_j \cap B_r(y)|\leq (1-\alpha_1)\pi r^2 \quad \textup{if }y\in \partial S_\ell^\delta \textup{ for }\ell = \ell_0 \textup{ or } \ell' 
\end{align}
due to \eqref{area density equation}. Also due to \eqref{area density equation} and \eqref{L1 convergence blowup},
\begin{align}\label{no third chamber}
    S_{\ell}^\delta \cap B_{r_j}= \emptyset\quad \forall \ell \notin \{\ell',\ell_0\}\,, \quad\textup{for large }j;
\end{align}
we may assume by restricting to the tail that \eqref{no third chamber} holds for all $j$. Next, a standard argument utilizing \eqref{L1 convergence blowup} and \eqref{area density equation blowup} implies that there exists $J\in \mathbb{N}$ such that for all $j\geq J$,  
\begin{align}\label{boundary trapped}
    (\partial S_{\ell_0}^\delta/r_j \cup \partial S_{\ell'}^\delta/ r_j) \cap Q \subset [-1,1]\times [-1/4,1/4]\,.
\end{align}
Now for almost every $t\in [-1,1]$, by Lemma \ref{slicing lemma}, the vertical slices (viewed as subsets of $\mathbb{R}$)
$$
(S_{\ell_0}^\delta/r_j)_t=S_{\ell_0}^\delta/r_j \cap Q \cap \{y:y\cdot e_1=t \}\,,\quad (S_{\ell'}^\delta/r_j)_t:=(S_{\ell'}^\delta)_t\cap Q \cap \{y:y\cdot e_1=t \}
$$ are one-dimensional sets of finite perimeter and, by \eqref{boundary trapped} and \cite[Proposition 14.5]{Mag12},
\begin{align}\notag
    2c_{\ell_0} + 2c_{\ell'}&\leq \int_{-1}^1 c_{\ell_0}P((S_{\ell_0}^\delta/r_j)_t;(-1,1))+c_{\ell'}P((S_{\ell'}^\delta/r_j)_t;(-1,1)) \,dt\\ \label{slicing perimeter estimate}
    &\leq c_{\ell_0}P(S_{\ell_0}^\delta/r_j;\mathrm{int}\,Q) + c_{\ell'}P(S_{\ell'}^\delta/r_j;\mathrm{int}\,Q)\,.
\end{align}
Since $\mathcal{H}^1(\partial \mathbb{H}^+ \cap \partial Q)=0$, \eqref{blow-up weak star} implies that
\begin{align}\label{weak star estimate}
    \lim_{j\to \infty }c_{\ell_0}P(S_{\ell_0}^\delta/r_j;\mathrm{int}\,Q) + c_{\ell'}P(S_{\ell'}^\delta/r_j;\mathrm{int}\,Q) = 2c_{\ell_0} + 2c_{\ell'}\,.
\end{align}
Together, \eqref{boundary trapped}-\eqref{weak star estimate} and Lemma \ref{slicing lemma} allow us to identify $j$ as large as we like (to be specified further shortly) and $1<t_1< t_2< 1 $ such that for $i=1,2$,
\begin{align}\label{one perimeter}
P((S_{\ell_0}^\delta/r_j)_{t_i};(-1,1))&=1=P((S_{\ell'}^\delta/r_j)_{t_i} ;(-1,1))\,,  \\ \notag
0&=\int_{(-1,1)}|\mathbf{1}_{(S_{\ell_0}^\delta/r_j)^+_{t_i}}-\mathbf{1}_{(S_{\ell_0}^\delta/r_j)^-_{t_i}}|+|\mathbf{1}_{(S_{\ell_0}^\delta/r_j)^+_{t_i}}- \mathbf{1}_{(S_{\ell_0}^\delta/r_j)_{t_i}}|\,d\mathcal{H}^1\\ \label{traces agree}
&=\int_{(-1,1)}|\mathbf{1}_{(S_{\ell'}^\delta/r_j)^+_{t_i}}-\mathbf{1}_{(S_{\ell'}^\delta/r_j)^-_{t_i}}|+|\mathbf{1}_{(S_{\ell'}^\delta/r_j)^+_{t_i}}- \mathbf{1}_{(S_{\ell'}^\delta/r_j)_{t_i}}|\,d\mathcal{H}^1\,,
\end{align}
where here and in the rest of the argument, the minus and plus superscripts denote left and right traces along $\{y\cdot e_1=t_i\}$ (again viewed as subsets of $\mathbb{R}$). From \eqref{boundary trapped} and \eqref{one perimeter}-\eqref{traces agree}, we deduce that there exist $-1/4\leq a_1 \leq b_1 \leq 1/4$ and $-1/4\leq a_2 \leq b_2 \leq 1/4$ such that
 \begin{align}\label{left and right traces}
    \mathcal{H}^1((S_{\ell_0}^\delta/r_j)^\pm_{t_i} \Delta [-1, a_i])=&0= \mathcal{H}^1((S_{\ell'}^\delta/r_j)^\pm_{t_i} \Delta [b_i, 1]) \quad\textup{ for }i=1,2\,.
\end{align}
Let us call $Q' = [t_1,t_2]\times [-1,1]$.
Since it will be useful later, we record the equality
\begin{align}\label{no perimeter on boundary cube}
  \mathcal{F}(\S^\delta) =   \mathcal{F}(S^\delta;\mathbb{R}^2\setminus r_jQ') + c_{\ell_0}P(S_{\ell_0}^\delta; \mathrm{int}\, r_jQ') + c_{\ell'}P(S_{\ell'}^\delta; \mathrm{int}\, r_jQ')\,,
\end{align}
which follows from \eqref{no third chamber}, \eqref{traces agree}, and Lemma \ref{slicing lemma}.
\par
Using the explicit description given by \eqref{boundary trapped} and \eqref{left and right traces}, we now identify a variational problem on $Q'$ for which our minimal partition must be optimal. We consider the minimization problem
\begin{align}\notag
    \inf_{\mathcal{A}_{Q'}} c_{\ell_0}P(A;\mathrm{int}\, Q') + c_{\ell'}P(B;\mathrm{int}\, Q')\,,
\end{align}
where
\begin{align}\notag
    \mathcal{A}_{Q'}&:=\{(A,B):A, B\subset Q',\, \restr{A}{\partial Q'}=S_{\ell_0}^\delta/r_j,\, \restr{B}{\partial Q'}=S^\delta_{\ell'}/r_j\textup{ in the trace sense},\, \\ \notag
    &\quad\qquad |A\cap B|=0,\,|A\cap Q'|=|(S_{\ell_0}^\delta/r_j)\cap Q'|,\,|B\cap Q'|=|(S_{\ell'}^\delta/r_j)\cap Q'|\}\,.
\end{align}
By the area constraint on elements in the class $\mathcal{A}_{Q'}$ and $\S^\delta \in \mathcal{A}_\delta^h$ or $\S^\delta \in \mathcal{A}_\delta^{\mm}$, any $\S$ given by
\begin{align}\notag
    S_{\ell_0} = (S_{\ell_0}\setminus r_jQ')\cup r_j(A \cap  Q')\,, \quad S_{\ell'}= (S_{\ell'}\setminus r_jQ')\cup r_j(B \cap  Q')\,, \quad  
    S_{\ell}= S_{\ell}^\delta\quad \ell\notin\{\ell_0,\ell'\}\,,
\end{align}
satisfies $|\mathbb{R}^2 \setminus \cup_\ell S_\ell| \leq \delta$ in the former case and $|\mathbb{R}^2 \setminus \cup_\ell S_\ell| \leq \delta$ and $(|S_1|,\dots,|S_N|)=\mm$ in the latter. Also, once $r_j$ is small enough, if $\S^\delta\in \mathcal{A}_\delta^h$, then $\S$ satisfies the trace condition \eqref{trace constraint} also. Therefore, $\S\in \mathcal{A}_\delta^h$ or $\S\in \mathcal{A}_\delta^{\mm}$, so we can compare
\begin{align}\notag
\mathcal{F}(\S^\delta)\hspace{.15cm} &\hspace{-.25cm}\overset{\eqref{no perimeter on boundary cube}}{=}   \mathcal{F}(\S^\delta;\mathbb{R}^2\setminus r_jQ') + c_{\ell_0}P(S_{\ell_0}^\delta; \mathrm{int}\, r_jQ') + c_{\ell'}P(S_{\ell'}^\delta; \mathrm{int}\, r_jQ') \\ \notag
&\leq \mathcal{F}(\S)= \mathcal{F}(\S^\delta;\mathbb{R}^2\setminus r_jQ') +  r_j c_{\ell_0}P(A; \mathrm{int}\, Q') + r_j c_{\ell'}P(B; \mathrm{int}\, Q')\,,
\end{align}
where in the last equality we have used the trace condition on $\mathcal{A}_{Q'}$ and the formula \eqref{trace difference} for computing $\mathcal{F}(\cdot;\partial Q')$. Discarding identical terms and rescaling, this inequality yields
\begin{align}\label{minimality on Q'}
    c_{\ell_0}P(S_{\ell_0}^\delta/r_j; \mathrm{int}\, Q') + c_{\ell'}P(S_{\ell'}^\delta/r_j; \mathrm{int}\, Q') \leq  c_{\ell_0}P(A; \mathrm{int}\, Q') +  c_{\ell'}P(B; \mathrm{int}\, Q')\,,
\end{align}
where $(A,B)\in \mathcal{A}_{Q'}$ is arbitrary. Simply put, our minimal partition must be minimal on $r_j Q'$ among competitors with the same traces and equal areas of all chambers.
\par
We now test \eqref{minimality on Q'} with a well-chosen competitor based on symmetrization. Let
\begin{align}\notag
    A=\{(x_1,x_2):-1\leq x_2 \leq \mathcal{H}^1((S_{\ell_0}^\delta/r_j)_{x_1})-1\} \,,\quad B= \{(x_1,x_2): 1\geq x_2 \geq 1-\mathcal{H}^1((S_{\ell'}^\delta)_{x_1})\}\,.
\end{align}
In the notation set forth in Lemma \ref{symmetrization lemma},
\begin{align}\notag
    A= (S_{\ell_0}^\delta/r_j)^h\,,\quad B= -(-S_{\ell'}^\delta/r_j)^h\,.
\end{align}
By \eqref{left and right traces} and \eqref{boundary trapped}, the assumptions of Lemma \ref{symmetrization lemma} are satisfied by $S_{\ell_0}^\delta/r_j$ and $ - S_{\ell'}^\delta/r_j$. Then the conclusions of that lemma imply that $(A,B) \in \mathcal{A}_{Q'}$, so \eqref{minimality on Q'} holds. 
However, \eqref{symmetrization inequality} also gives
\begin{align}\label{reverse minimality on Q'}
    c_{\ell_0}P(S_{\ell_0}^\delta/r_j; \mathrm{int}\, Q') + c_{\ell'}P(-S_{\ell'}^\delta/r_j; \mathrm{int}\, Q') \geq  c_{\ell_0}P(A; \mathrm{int}\, Q') +  c_{\ell'}P(-B; \mathrm{int}\, Q')\,,
\end{align}
so that in fact there is equality. But according to Lemma \ref{symmetrization lemma}, every vertical slice of $(S_{\ell_0}^\delta/r_j)\cap Q'$ and $(-S_{\ell'}^\delta/r_j)\cap Q'$ must therefore be an interval with one endpoint at $-1$. This is precisely what we set out to prove in this step.
\par
\medskip
\noindent\textit{Step two}: Here we prove that for the open set $G^{\delta}$ (see Remark \ref{lebesgue representative remark}), the set
\begin{align}\notag
    \mathcal{I}:=\{t\in [r_jt_1/2,r_jt_2/2]: (G^\delta \cap r_jQ')_t = \emptyset \}
\end{align}
is a closed interval. $\mathcal{I}$ is closed since the projection of the open set $G^\delta \cap r_j Q'$ onto the $x_1$ axis is open, so we only need to prove it is an interval. First, we claim that for any rectangle $R'=(T_1,T_2)\times [-r_j,r_j]$ with $(T_1,T_2)\subset \mathcal{I}^c$,
\begin{align}\label{graphicality claim}
\mbox{$\partial S_{\ell_0}^\delta \cap  R$ and $\partial S_{\ell'}^\delta \cap  R$ are graphs of functions $F_0$ and $F'$}
\end{align}
with $F_0<F'$, over the $x_1$-axis of constant curvature with no vertical tangent lines in $R'$. To see this, first note that for any $(a,b)\subset\!\subset (T_1,T_2)$, $ \partial S_{\ell_0}^\delta \cap ((a,b) \times [-r_j,r_j])$ and $\partial S_{\ell'}^\delta \cap ((a,b) \times [-r_j,r_j])$ must be at positive distance from each other by the definition of $\mathcal{I}^c$. Then a first variation argument implies that each has constant mean curvature in the distributional sense, and a graph over $(a,b)$ with constant distributional mean curvature must be a single arc of constant curvature with no vertical tangent lines in the interior. Letting $(a,b)$ exhaust $(T_1,T_2)$, we have proven the claim.
\par
Suppose for contradiction that there exist $T_i\in \mathcal{I}$, $i=1,2$, such that $(T_1,T_2)\cap \mathcal{I}^c\neq \emptyset$. Set $(T_1,T_2)\times [-r_j,r_j]= R$. Now $F_0$ and $F_1$ extend continuously to $T_1$ and $T_2$ with $F_0(T_i) \leq F'(T_i)$ for each $i$. In fact $F_0(T_i)=F'(T_i)$. If instead we had for example $F_0(T_1)<F'(T_1)$, then $G^\delta$ would contain a rectangle $(t,T_1)\times (c,d)$ for some $t<T_1$ and $c<d$, which would imply that $G^\delta$ has positive density at $(T_1, F_0(T_1))$ and $(T_1,F'(T_1))$. By Corollary \ref{no termination corollary}, $\partial G^\delta \cap \partial S_{\ell_0}^\delta$ is single arc of constant curvature in neighborhood $N$ of $(T_1,F_0(T_1))$, which, by $T_1\in \mathcal{I}$, has vertical tangent line at $(T_1,F_0(T_1))$. Therefore, $ \partial S_{\ell_0}^\delta \cap N \cap R$ is either a vertical segment or a circle arc with vertical tangent line at $(T_1,F_0(T_1))$, and both of these scenarios contradict \eqref{graphicality claim}. So we have $F_0(T_i)=F'(T_i)$, and thus $(T_i,F_0(T_i))\in \partial S_{\ell_0}^\delta \cap \partial S_{\ell'}^\delta \cap \partial G^\delta$. As a consequence, by Corollary \ref{no termination corollary}, $G^\delta$ must have density $0$ at $(T_i,F_0(T_i))$, which means that the graphs of $F_0$ and $F'$ meet tangentially at $T_i$. But the only way for two circle arcs to meet tangentially at two common points is if they are the same arc, that is $F_0=F'$, which is a contradiction of $F_0<F'$. We have thus shown that $\mathcal{I}$ is a closed interval.
\par
\medskip
\noindent\textit{Step three}: Finally we may finish the proof. We note that by our assumption $0 \in \partial S_{\ell'}^\delta \cap \partial S_{\ell_0}^\delta$ (see the beginning of step one), $0\in I$. Now if $0\in \mathrm{int}\, I$, then $|G^\delta \cap B_{r}(0)|=0$ for some small $r$, and we have $(ii)$. If $\{0\}=I$, then by the same argument as at the beginning of the previous step, we know that $\partial S_{\ell_0}^\delta/r_j \cap (Q' \setminus \{0\})$ and $\partial S_{\ell'}^\delta/r_j \cap (Q'\setminus \{0\})$ are each two circle arcs of equal curvature meeting at the origin. Furthermore, since the blow-up of $G^\delta$ is empty at $0$, we see that all four of these arcs must meet tangentially at the origin, so that $(iii)$ holds. Lastly, if $\mathrm{int}\,I\neq \emptyset$ and $0\in \partial I$, the combined arguments of the previous two cases imply that $(iv)$ holds.\end{proof}

\begin{theorem}[Boundary Resolution for $\delta>0$]\label{boundary resolution theorem for positive delta}
    If $\S^\delta$ minimizes $\mathcal{F}$ among $\mathcal{A}_\delta^h$ for some $\delta>0$ and $x\in \partial S_{\ell_0}^\delta \cap \partial B$, then there exists $r_x>0$ such that exactly one of the following is true:
\begin{enumerate}[label=(\roman*)]
\item $x$ is not a jump point of $h$ and $B_{r_x}(x) \cap B = B_{r_x}(x) \cap  S_{\ell_0}^\delta$;
\item $x$ is a jump point of $h$ and $\partial S_{\ell_0}^\delta \cap B_{r_x}(x)$ is a line segment separating $B_{r_x}(x) \cap B$ into $S_{\ell_0}^\delta \cap B_{r_x}(x) \cap B$ and $S_{\ell'}^\delta \cap B_{r_x}(x) \cap B$ for some $\ell'\neq \ell_0$;
\item $x$ is a jump point of $h$, and there exists circle arcs $a_1$ and $a_2$ meeting at $x$ such that
\begin{align}\notag
    \partial S_{\ell_0}^\delta\cap \partial G^\delta \cap B_{r_x}(x)= a_1\,,\quad \partial S_{\ell'}^\delta\cap \partial G^\delta \cap B_{r_x}(x)= a_2\,, \quad \partial S_{\ell_0}^\delta \cap \partial S_{\ell'}^\delta \cap B_{r_x}(x)=\{x\}\,.
\end{align}
\end{enumerate}
\end{theorem}

\begin{proof}
    Let us assume for simplicity that $x=\vec{e}_1$. The proof at any other point in $\partial B$ is the same.
\par
\medskip
\noindent\textit{Step zero}: If $\vec{e}_1\in \partial B$ is not a jump point of $h$, then by the inclusion \eqref{circular segment containment} from Theorem \ref{existence theorem}, $(i)$ holds. 
\par
\medskip
\noindent\textit{Step one}: For the rest of the proof, we assume that $\vec{e}_1$ is a jump point of $h$. By Theorem \ref{positive delta classification theorem}, there exists $\ell'\neq \ell_0$ such that any blow-up at $\vec{e}_1$ consists of the blow-up chambers $S_{\ell_0}$, $S_{\ell'}$, each of which is the intersection of a halfspace with $\{y:y\cdot \vec{e}_1<0\}$, $S_0=\{y:y\cdot \vec{e}_1>0\}$, and $G = \mathbb{R}^2\setminus (S_0\cup S_{\ell_0}\cup S_{\ell'})$ is a possibly empty connected cone contained in $\{y:y\cdot \vec{e}_1<0\}$. In this step we argue that on a small rectangle $Q'$ with $0\in \partial Q'$, $( S_{\ell_0}-\vec{e}_1)/r_j \cap Q'$ and $( S_{\ell'}^\delta - \vec{e}_1)/r_j \cap Q'$ are the hypograph and epigraph, respectively of two functions over $\{y\cdot \vec{e}_1=0\}$.
\par
Let us choose $r_j \to 0$ such that by Theorem \ref{positive delta classification theorem}, we have a blow-up limit belonging to $\mathcal{A}_{\vec{e}_1}$. By the density estimates \eqref{area density equation}, $B_{r_j}(x) \subset S_{\ell_0}^\delta \cup S_{\ell'}^\delta \cup G^\delta \cup S_0$ for all large enough $j$, so we can ignore the other chambers. Also, for convenience, by the containment \eqref{circular segment containment} of the circular segments in $S_{\ell_0}^\delta$ and $S_{\ell'}^\delta$ from Theorem \ref{existence theorem}, we extend $S_{\ell_0}^\delta$ and $S_{\ell'}^\delta$ on $\{y:y\cdot \vec{e}_1<1\}$ so that for all large $j$,
$$
\{y:y\cdot \vec{e}_1 = 1\}\cap B_{r_j}(\vec{e}_1) \subset  \partial S_{\ell_0}^\delta \cup \partial S_{\ell'}^\delta
$$
rather than
$$
\partial B\cap B_{r_j}(\vec{e}_1) \subset  \partial S_{\ell_0}^\delta \cup \partial S_{\ell'}^\delta\,;
$$
this allows us to work on a rectangle along the sequence of blow-ups rather than $(B - \vec{e}_1)/r_j$. Now due to the inclusion \eqref{circular segment containment} from Theorem \ref{existence theorem}, there exists a rectangle
\begin{align}\notag
    Q = [T,0]\times [-1,1]
\end{align}
such that for all large $j$, up to interchanging the labels $\ell_0$ and $\ell'$, in the trace sense, 
\begin{align}\notag
    (\{T\}\times [-1, -1/2]) \cup ([T,0]\times \{-1\}) \cup (\{0\}\times [-1,0]) &\subset  ( S_{\ell_0}^\delta - \vec{e}_1)/r_j\,, \\ \notag
    (\{T\}\times [1/2, 1]) \cup ([T,0]\times \{1\}) \cup (\{0\}\times [0,1]) &\subset  ( S_{\ell'}^\delta - \vec{e}_1)/r_j\,.
\end{align}
Then a similar slicing argument as leading to \eqref{left and right traces} implies that for some large $j$, there exist $-1/2\leq a_1\leq a_2 \leq 1/2$ and $t\in [T, 0)$ such that, in the trace sense, 
\begin{align}\notag
    (\{t\}\times [-1, a_1]) \cup ([t,0]\times \{-1\}) \cup (\{0\}\times [-1,0]) &=  ( S_{\ell_0}^\delta - \vec{e}_1)/r_j \\ \notag
    (\{t\}\times [a_2, 1]) \cup ([t,0]\times \{1\}) \cup (\{0\}\times [0,1]) &=  ( S_{\ell'}^\delta - \vec{e}_1)/r_j\,.
\end{align}
Given this explicit description on the boundary of $Q':=[t,0]\times [-1,1]$, the same argument as in the proof of Theorem \ref{interior resolution theorem delta positive} gives claim of this step.
\par
\medskip
\noindent\textit{Step two}: We may finally finish the proof of Theorem \ref{boundary resolution theorem for positive delta}. By the same argument as in the previous theorem, the set
\begin{align}\notag
    \mathcal{I}:=\{s\in [r_jt,1]: (G^{\delta} \cap (r_jQ'+\vec{e}_1))_s=\emptyset \}
\end{align}
is a closed interval. Furthermore, since $\vec{e}_1$ is a jump point of $h$, $\mathcal{I}$ contains $0$. If $\mathrm{int}\, I \neq \emptyset$, we immediately see that $(ii)$ holds. On the other hand, if $\mathcal{I}=\{0\}$, then the vertical slices of $G^\delta$ are non-empty for all $s\in (r_jt,1)$. Again the same argument as in the previous theorem shows that $(iii)$ holds. 
\end{proof}

\begin{proof}[Proof of Theorem \ref{main regularity theorem delta positive}]
At any $x\in \cl B$, Theorems \ref{interior resolution theorem delta positive} and \ref{boundary resolution theorem for positive delta} yield the existence of $r_x>0$ such that either $x$ is an interior point of $S_\ell^\delta$ or $G^\delta$ or on $B_{r_x}(x)$, the minimizer is described by one of the options listed in those theorems. By the enumeration of possible local resolutions in those theorems, we see that $\partial S_\ell^\delta \cap B$ is $C^{1,1}$ as desired, since it is analytic except where two arcs of constant curvature intersect tangentially. Now if $x$ and $y$ are both in $\partial G^\delta \cap \partial S_\ell^\delta$ for some $\ell\geq 1$, then one of Theorem \ref{interior resolution theorem delta positive}.$(i)$, $(iii)$, or $(iv)$ or Theorem \ref{boundary resolution theorem for positive delta}.$(iii)$ holds on $B_{r_x}(x)$ and $B_{r_y}(y)$; in particular, each $\partial G^\delta \cap \partial S_\ell^\delta \cap B_{r_x}(x)$ and $\partial G^\delta \cap \partial S_\ell^\delta \cap B_{r_y}(y)$ is an arc of constant curvature. A first variation argument then gives \eqref{curvature condition} if $G^\delta \neq \emptyset$. Also, by the compactness of $\cl B$ and the interior resolution theorem, there are only finitely many arcs in $\partial G^\delta \cap \partial S_\ell^\delta$. We note that $H_{S_\ell^\delta}$ cannot be negative along $\partial^* S_\ell^\delta \cap \partial^* G^\delta$, since local variations which decrease the area of $G^\delta$ are admissible. A similar argument based on the interior local resolution result implies that if $\mathcal{H}^1(\partial S_\ell \cap \partial S_m)>0$ for $\ell,m\geq 1$, then $\partial S_{\ell}^\delta \cap \partial S_{m}^\delta$ is composed of finitely many straight line segments. We have thus decomposed each such $\partial S_\ell^\delta \cap \partial S_m^\delta$ and $\partial G^\delta \cap \partial S_\ell^\delta$ into finitely many line segments and arcs of constant curvature, respectively.
\par
Moving on to showing that each connected component, say $C$, of $S_\ell^\delta$ for $1\leq \ell \leq N$ is convex, consider any $x\in \partial C$. $C \cap B_{r_x}(x)$ must be convex by Theorems \ref{interior resolution theorem delta positive} and \ref{boundary resolution theorem for positive delta} and $H_{S_\ell^\delta}\geq 0$ along $\partial^* S_\ell^\delta \cap \partial^* G^\delta$.  Since $\partial C$ consists of a finite number of segments and circular arcs and $C$ is connected, the convexity of $C$ follows from this local convexity. To finish proving the theorem, it remains to determine the ways in which these line segments and arcs may terminate. We note that each component of $\partial G^\delta$ must terminate. If one did not, then by Corollary \ref{no termination corollary}, it forms a circle contained in $\partial S_\ell^\delta \cap \partial G^\delta$. This configuration cannot be minimal however, since that component of $G^\delta$ may be added to $S_\ell^\delta$ to decrease the energy. Suppose that one of these components terminates at $x$. Next, by applying the local resolution at $x$, either Theorem \ref{interior resolution theorem delta positive}.$(iv)$ holds if $x\in B$ or item $(ii)$ or $(iii)$ from Theorem \ref{boundary resolution theorem for positive delta} holds, where $x\in \partial B$ is a jump point of $h$. This yields the desired conclusion.\end{proof}

\begin{proof}[Proof of Theorem \ref{main regularity theorem delta positive all of space}]
    The proof is similar to the proof of Theorem \ref{main regularity theorem delta positive}. Since every interface point is an interior interface point, determining the ways in which arcs may terminate proceeds as in the case $x\in B$ in that theorem.
\end{proof}

\section{Proof of Theorem \ref{main regularity theorem delta zero}}\label{sec:proof of delta 0 theorem}

\begin{proof}[Proof of Theorem \ref{main regularity theorem delta zero}]
    \textit{Step one}: First, we show that the set $\Sigma$ of interior triple junctions, or more precisely the set
\begin{align}\notag
    \Sigma := \{x\in B : \textup{$\exists$ a blow-up at $x$ given by $(ii)$ from Theorem \ref{zero delta classification theorem}}\}
\end{align}
does not have any accumulation points in $\cl B$. Suppose for contradiction that $\{x_k\}$ is a sequence of such points accumulating at $x\in \cl B$. By restricting to a subsequence and relabeling the chambers, we can assume that the three chambers in the blow-ups at each $x_k$ are $S^0_\ell$ for $1\leq \ell \leq 3$. In both cases $x\in B$ and $x\in \partial B$, the argument is similar (and follows classical arguments, e.g. \cite[Theorem 30.7]{Mag12}), so we consider only the case where $x\in \partial B$. If $\{x_k\}\subset \Sigma$ and $x_k \to x\in \partial B$, then by \eqref{circular segment containment}, $x\in \partial B$ is a jump point of $h$. We claim that up to a subsequence which we do not notate,
\begin{align}\label{local L1 blowup conv}
    \frac{S_\ell^0-x}{|x-x_k|} \to S_{\ell}\quad\textup{ locally in $L^1$ for $\ell=1,2,3$}
\end{align}
for a blow-up cluster $\S$ of the form from item $(v)$ in Theorem \ref{zero delta classification theorem}. To see this, we first note that by our assumption on $x_k$,
\begin{align}\label{in all three boundaries}
    x\in \partial S_1^0 \cap \partial S_2^0 \cap \partial S_3^0\,.
\end{align}
This inclusion rules out item $(iv)$ from Theorem \ref{zero delta classification theorem}, and so the blow-up cluster is three connected cones partitioning $\{y:y\cdot x<0\}$. Up to a further subsequence, we may assume that
\begin{align}\notag
    \frac{x_k - x}{|x_k - x|}\to \nu \in \{y:y\cdot x < 0\}\,,
\end{align}
where we have used \eqref{circular segment containment} to preclude the possibility that $x_k$ approaches $x$ tangentially. Now for some $r>0$, $B_r(\nu)$, and $\ell_0\in \{1,2,3\}$, say $\ell_0=1$, the description of the blow-up cluster implies that $B_r(\nu) \subset S_2\cup S_3$. Combined with the $L^1$ convergence \eqref{local L1 blowup conv} and the infiltration lemma, we conclude that $B_{|x_k-x|r/4}\subset S_2^0 \cup S_3^0$ for large enough $k$, which is in direct conflict with $x_k \in \Sigma$. We have thus proven that $\Sigma$ has no accumulation points in $\cl B$; in particular, it is finite.
\par
\medskip
\noindent\textit{Step two}: We finally conclude the proof of Theorem \ref{main regularity theorem delta zero}. For any $x\in (B\setminus \Sigma) \cap \partial S_{\ell_0}^0$, Theorem \ref{zero delta classification theorem} and the infiltration lemma imply that $x\in \partial^* S_{\ell_0}^0 \cap \partial^* S_{\ell_1}^0$. In turn, by Corollary \ref{reduced boundary regularity delta 0}, there exists $r_x>0$ such that $B_{r_x}(x) \cap \partial S_{\ell_0}^0$ is a diameter of $B_{r_x}(x)$. Recalling from Corollary \ref{density corollary} that $\mathcal{H}^1(\partial S_{\ell_0}^0 \setminus \partial^* S_{\ell_0}^0)=0$, we may thus decompose $\partial S_{\ell_0}$ as a countable number of line segments, each of which must terminate at a point in the finite set $\Sigma$ or a jump point of $h$.  Therefore, $\partial S_{\ell_0}^\delta$ is a finite number of line segments. The remainder of Theorem \ref{main regularity theorem delta zero} now follows directly from this fact and the classification of blow-ups in Theorem \ref{zero delta classification theorem}, items $(ii)$, $(iv)$, and $(v)$. Indeed, since the interfaces are a finite number of line segments, at $x\in \Sigma$ or $x\in \partial B$ which is a jump point of $h$, the blow-up is unique, and the minimal partition $\S^0$ must coincide with the blow-up on a neighborhood of $x$. The convexity of connected components of $S_\ell^0$ for $1\leq \ell \leq N$ follows as in the $\delta>0$ case.
\end{proof}


\section{Resolution for small $\delta$ on the ball}\label{sec:resolution for small delta}

\begin{proof}[Proof of Theorem \ref{resolution for small delta corollary}]
\textit{Step zero}: We begin by reducing the statement of the theorem to one phrased in terms of a sequence of minimizers $\{\S^{\delta_j}\}$. More precisely, to prove Theorem \ref{resolution for small delta corollary}, we claim it is enough to consider a sequence $\{\S^{\delta_j}\}$ of minimizers for $\delta_j\to 0$ and show that up to a subsequence, there exists a minimizer $\S^0$ among $\mathcal{A}_{0}^h$ with singular set $\Sigma$ such that 
\begin{align}\label{hausdorff convergence of chambers in proof}
\max \big\{ \sup_{x\in S_\ell^{\delta_j}} \dist (x,S_\ell^0)\,,\,\sup_{x\in S_\ell^{0}} \dist (x,S_\ell^{\delta_j}) \big\} &\to 0\quad \textit{for }1\leq \ell \leq N \\ 
 \label{hausdorff convergence of remnants in proof}
\max \big\{ \sup_{x\in G^{\delta_j}} \dist (x,\Sigma)\,,\,\sup_{x\in \Sigma} \dist (x,G^{\delta_j}) \big\} &\to 0 \,,
\end{align}
and, for large enough $j$ and each $x\in \Sigma$, $B_{r}(x) \cap \partial G^{\delta_j}$ consists of three circle arcs of curvature $\kappa_{j}$, with total area $|G^{\delta_j}|=\delta_j$. To see why this is sufficient, if Theorem \ref{resolution for small delta corollary} were false, then there would be some sequence $\delta_j\to 0$ with minimizers $\S^{\delta_j}$ among $\mathcal{A}_{\delta_j}^h$ such that for any subsequence and choice of minimizer $\S^0$ among $\mathcal{A}_0^h$, at least one of \eqref{hausdorff convergence of chambers in corollary}-\eqref{hausdorff convergence of remnants in corollary} or the asymptotic resolution near singularities of $\S^0$ did not hold. 
But this would contradict the subsequential claim above. 
\par
We point out that if we knew that $\partial G^\delta$ is described near singularities by three circle arcs for small $\delta$, the saturation of the area inequality $|G^\delta|\leq \delta$ follows from the facts that $\partial G^\delta$ has negative mean curvature away from its cusps and increasing the area of $G^\delta$ is admissible if $|G^\delta|<\delta$. Therefore, the rest of the proof is divided into steps proving \eqref{hausdorff convergence of chambers in proof}-\eqref{hausdorff convergence of remnants in proof} and the asymptotic resolution near singular points. First we prove that due to $c_\ell=1$ for $1\leq \ell \leq N$, there are no ``islands" inside $B$. Second, we extract a minimizer $\S^0$ for $\mathcal{A}_0^h$ from a minimizing (sub-)sequence $\S^{\delta_j}$ with $\delta_j \to 0$ and prove \eqref{hausdorff convergence of chambers in proof}. There are then two cases. In the first, we suppose that the set of triple junctions $\Sigma$ is empty and show that $G^{\delta_j}=\emptyset$ for large $j$, so that \eqref{hausdorff convergence of remnants in proof} is trivial. In the other case, we assume that $\Sigma \neq \emptyset$ and prove \eqref{hausdorff convergence of remnants in proof} and the final resolution near singularities of the limiting cluster.
\par
\medskip
\noindent\textit{Step one}: Let $\S^\delta$ be a minimizer for $\delta>0$. We claim that for any connected component $C$ of any chamber $S_\ell^\delta$ with $1\leq \ell \leq N$, $\partial C \cap \{h=\ell\}\neq \emptyset$. Suppose that this were not the case for some $C\subset S_\ell$. Then in fact, $\cl C \subset B$, since by Theorem \ref{boundary resolution theorem for positive delta} and \eqref{circular segment containment}, the only components that can intersect $\partial B$ are those bordering $\partial B$ along an arc where $h=\ell$. By Theorem \ref{interior resolution theorem delta positive}, $\partial C$ is $C^{1,1}$ since its boundary is contained in $B$. If $\mathcal{H}^1(\partial C \cap \partial S_{\ell'}^\delta)>0$ for some $\ell'$, then since all $c_\ell$ are equal, removing $C$ from $S_\ell^\delta$ and adding it to $S_{\ell'}^\delta$ contradicts the minimality of $\S^\delta$. So it must be the case that $\partial C \subset \partial G^\delta$ except for possibly finitely many points. We translate $C$ if necessary until it intersects $\partial G^\delta \cap \partial C'$ for a connected component $C'\neq C$ of some $S_{\ell'}^\delta$ at $y\in B$, which does not increase the energy. Creating a new minimal cluster $\tilde{\S}$ by adding $C$ to $S_{\ell'}^\delta$ and removing it from $S_\ell^\delta$ gives a contradiction. This is because by Corollary \ref{density corollary}, $y\in (\tilde{S}_{\ell'}^\delta)^{(1)}$ implies that $y\in\mathrm{int}\,(\tilde{S}_{\ell'}^\delta)^{(1)}$, and so $\mathcal{F}(\tilde{\S};B_r(y))=0$ for some $r>0$, against the minimality of $\S^\delta$.

\par
We note that as a consequence, the total number of connected components in $\S^\delta$ is bounded in terms of the number of jumps of $h$, and in addition the area of any connected component is bounded from below by the area of the smallest circular segment from \eqref{circular segments}.
\par
\medskip
\noindent\textit{Step two}: Here we identify our subsequence, limiting minimizer among $\mathcal{A}_0^h$, and prove \eqref{hausdorff convergence of chambers in proof}. Let us decompose each $S_{\ell}^{\delta_j}$ into its open connected components
\begin{align}
    S_\ell^{\delta_j} = \cup_{i=1}^{N_\ell^{j}} C_i^{\ell,j}\,,
\end{align}
where by the previous step, $N_\ell^j\leq N_\ell(h)$ for all $j$ and $|C_i^{\ell,j}|\geq C(h)$ for all $j$ and $i$. Up to a subsequence which we do not notate, we may assume therefore that for each $1\leq \ell\leq N$,
\begin{align}\label{number and area bound equation}
    N_\ell^j = M_\ell\leq N_\ell(h)\quad \textup{and}\quad |C_i^{\ell,j}|\geq C(h)\quad  \forall j \quad\textup{and} \quad i\in \{1,\dots, M_\ell\}\,.
\end{align}
Since 
\begin{align}\label{ordering of infimums}
\min_{\mathcal{A}_{\delta_j}^h}\mathcal{F} \leq \min_{\mathcal{A}_{0}^h}\mathcal{F}\quad \forall j\,,
\end{align}
up to a further subsequence, the compactness for sets of finite perimeter and \eqref{number and area bound equation} yield a partition $\{C_i^\ell\}_{\ell,i}$ of $B$, with no trivial elements thanks to \eqref{number and area bound equation}, such that 
\begin{align}\label{almost everywhere convergence}
    \mathbf{1}_{C_i^{\ell,j}} &\to \mathbf{1}_{C_i^\ell}\quad\textup{a.e. }\quad\textup{and} \\ \label{lsc delta to 0}
    \liminf_{j\to \infty}\mathcal{F}(\S^{\delta_j};B)= \liminf_{j\to \infty}&\sum_{\ell=1}^N\sum_{i=1}^{M_\ell}  P(C_i^{\ell,j};B) \geq \sum_{\ell=1}^N\sum_{i=1}^{M_\ell}  P(C_i^{\ell};B)\quad \forall 1\leq \ell\leq N.
\end{align}
Actually, by Lemma \ref{convex sets convergence lemma}, we may assume that each $\cl\, C_i^\ell$ is compact and convex, $C_i^\ell$ is open, and, for each $1\leq \ell\leq N$,
\begin{align}\label{hausdorff convergence of components delta to 0}
\max \big\{ \sup_{x\in C_i^\ell} \dist (x,C_{i}^{\ell,j})\,,\,\sup_{x\in C_i^{\ell,j}} \dist (x,C_i^\ell) \big\} \to 0 \quad \forall 1\leq i \leq M_\ell\,.
\end{align}
We claim that the cluster
\begin{align}\notag
\S^0=(\mathbb{R}^2\setminus B,S_1^0,\dots,S_N^0, \emptyset)=\Big(\mathbb{R}^2\setminus B,\bigcup_{i=1}^{M_1}C_i^1,\dots,\bigcup_{i=1}^{M_N}C_i^N,\emptyset\Big)
\end{align}
of $B$ is minimal for $\mathcal{F}$ on $\mathcal{A}_0^h$. It belongs to $\mathcal{A}_0^h$ by the inclusion \eqref{circular segment containment} for each $j$ and by $\delta_j\to 0$. For minimality, we use \eqref{ordering of infimums} and \eqref{lsc delta to 0} to write
\begin{align}\label{minimality of thing}
    \min_{\S\in\mathcal{A}_0^h}\mathcal{F}(\S;B) \geq   \sum_{\ell=1}^N\liminf_{j\to \infty}\sum_{i=1}^{M_\ell}  P(C_i^{\ell,j};B) \geq \sum_{\ell=1}^N\sum_{i=1}^{M_\ell}  P(C_i^{\ell};B) \geq \sum_{\ell=1}^N  P(S_\ell^0;B)\,.
\end{align}
This proves $\S^0$ is minimal. The Hausdorff convergence \eqref{hausdorff convergence of chambers in proof} follows from \eqref{hausdorff convergence of components delta to 0}. 
\par
We note that by the minimality of $\S^0$, \eqref{minimality of thing} must be an equality, so that in turn
\begin{align}\label{perimeter equality}
    \sum_{i=1}^{M_\ell} P(C_i^{\ell};B) =  P(S_\ell^0;B)\quad\forall 1\leq \ell\leq N\,.
\end{align}
Now each $C_i^\ell$ is open and convex; in particular, they are all indecomposable sets of finite perimeter. This indecomposability and \eqref{perimeter equality} allow us to conclude from \cite[Theorem 1]{AMMN01} that $\{C_i^\ell\}_{i}$ is the unique decomposition of $S_\ell^0$ into pairwise disjoint indecomposable sets such that \eqref{perimeter equality} holds. Also, by Theorem \ref{main regularity theorem delta zero}, each $(S_\ell^0)^{(1)}$ is an open set whose boundary is smooth away from finitely many points. By \cite[Theorem 2]{AMMN01}, which states that for an open set with $\mathcal{H}^1$-equivalent topological and measure theoretic boundaries (e.g.\ $(S_\ell^0)^{(1)}$) the decompositions into open connected components and maximal indecomposable components coincide, we conclude that the connected components of $(S_\ell^0)^{(1)}$ are $\{C_i^\ell\}_{i=1}^{M_\ell}$, and $S_\ell^0=(S_\ell^0)^{(1)}$. We have in fact shown in \eqref{hausdorff convergence of components delta to 0} that the individual connected components of $S_\ell^{\delta_j}$ converge in the Hausdorff sense to the connected components of $S_\ell^0$ for each $\ell$.
\par
\medskip
\noindent\textit{Step three}: In this step, we suppose that $\Sigma=\emptyset$ and show that $G^{\delta_j}=\emptyset$ for large $j$, which finishes the proof in this case. If $\Sigma=\emptyset$, then every component of $\partial S_\ell^0 \cap \partial S_{\ell'}^0$ is a segment which, by Theorem \ref{main regularity theorem delta zero}, can only terminate at a pair of jump points of $h$ which are not boundary triple junctions. Therefore, every connected component of a chamber $S_\ell^0$ is the convex hull of some finite number of arcs on $\partial B$ contained in $\{h=\ell\}$. Now for large $j$, by the Hausdorff convergence in step two and the containment \eqref{circular segment containment}, given any connected component $C$ of a chamber of $\S^{\delta_j}$ there exists connected component $C'$ of a chamber of $\S^0$ such that $\partial C \cap \partial B = \partial C' \cap \partial B$. Since every connected component of every chamber is convex for $\delta \geq 0$, we see that in fact it must be $C=C'$. So the minimal partition $\S^{\delta_j}$ coincides with $\S^0$ for all large $j$ when there are no triple junctions of $\S^0$.
\par
\medskip
\noindent\textit{Step four}: For the rest of the proof, we assume that $\Sigma \neq \emptyset$. In this step, we show that 
\begin{align}\label{diverging curvatures}
G^{\delta_j}\neq \emptyset \quad\textup{for all $j$}\quad\textup{and}\quad  \kappa_{j}\to \infty\,.
\end{align}
Assume for contradiction that $G^{\delta_j}= \emptyset$ for some $j$. Then $\S^{\delta_j}$ is minimal for $\mathcal{F}$ among $\mathcal{A}_0^h$, so $\mathcal{F}(\S^{\delta_j})= \mathcal{F}(\S^0)$ and $\S^0$ is minimal among $\mathcal{A}_h^{\delta_j}$, too. But this is impossible, since $\Sigma\neq \emptyset$ and Theorem \ref{main regularity theorem delta positive} precludes the presence of interior or boundary triple junctions for minimizers when $\delta>0$. Moving on to showing that $ \kappa_{j}\to \infty$, we fix $y\in \Sigma$. Let us assume that $y\in \partial B$ is a jump point of $h$ between $h=1$ and $h=2$ with $S_3^0$ being the third chamber in the triple junction, since the case when $y\in B$ is easier. For all $j$, by the containment \eqref{circular segment containment} of the neighboring circular segments in $S_1^{\delta_j}$ and $S_2^{\delta_j}$, there exists $r>0$ such that for all $j$ and $3\leq \ell \leq N$, $\partial S_\ell^{\delta_j} \cap B_r(y) \subset B$ for some small $r$. In particular, $\partial S_3^{\delta_j} \cap B_r(y)$ is $C^{1,1}$ by Theorem \ref{main regularity theorem delta positive}. Furthermore, since $S_3^{\delta_j}$ converges as $j\to \infty$ to a set with a corner in $B_r(y)$, the $C^{1,1}$ norms of $\partial S_3^{\delta_j}$ must be blowing up on that ball. These norms are controlled in terms of $\kappa_j$, and so $\kappa_{j}\to \infty$. 
\par
\medskip
\noindent\textit{Step five}: In the next two steps, we prove \eqref{hausdorff convergence of remnants in proof}. Here we show that 
\begin{align}\label{first half of hausdorff convergence}
    \sup_{x\in G^{\delta_j}} \dist (x, \Sigma) \to 0\,.
\end{align}
Suppose for contradiction that \eqref{first half of hausdorff convergence} did not hold. Then, up to a subsequence, we could choose $r>0$ and $y_j \in \cl\, G^{\delta_j}$ such that 
$$
y_j \to y\in \cl B \setminus \cup_{z\in \Sigma}B_r(z)\,.
$$
Let us assume that $y=\vec{e}_1\in \partial B$; we will point out the difference in the $y\in B$ argument when the moment arises. We note that $y$ must be a jump point of $h$, say between $h=1$ and $h=2$, due to \eqref{circular segment containment}. Furthermore, by Theorem \ref{main regularity theorem delta zero} and $y\notin \Sigma$, there exists $r'>0$ such that
$$
B_{r'}(y) \cap B \subset \cl\, S_1^0 \cup \cl\, S_2^0\,.
$$
In particular, $\dist(y,S_\ell^0)>r'/2$ for $3\leq \ell \leq N$. Therefore, due to \eqref{hausdorff convergence of chambers in proof}, $\dist (y, S_\ell^{\delta_j})\geq r'/2$ for large enough $j$. Also by \eqref{circular segment containment} applied to $S_1^{\delta_j}$ and $S_2^{\delta_j}$ and the convexity of connected components of those sets, we may choose small $\varepsilon_1$ and $\varepsilon_2$ such that on the rectangle
\begin{align*}
    R = [1-\varepsilon_1,1]\times [-\varepsilon_2,\varepsilon_2] \subset B_{r'/2}(y)\,,
\end{align*}
$\partial S_1^{\delta_j} \cap R \cap B$ and $\partial S_2^{\delta_j}\cap R \cap B$ are graphs of functions $f_1^j$ and $f_2^j$ over the $\vec{e}_1$-axis for all $j$. Relabeling if necessary, we may take
\begin{align}\notag
    -\varepsilon_2 \leq f_1^j \leq f_2^j \leq \varepsilon_2\quad \textup{and}\quad  (f_1^j)''\leq 0\,,\,\, (f_2^j)'' \geq 0\,.
\end{align}
It is at this point that in the case $y\in  B$, we instead appeal to the Hausdorff convergence \eqref{hausdorff convergence of chambers in proof} and the convexity of the components of $S_\ell^{\delta_j}$ to conclude that graphicality holds. Now the set
\begin{align}\notag
    \mathcal{I}_j=\{t\in [1-\varepsilon_1,1]:f_1^j=f_2^j\}
\end{align}
is a non-empty interval by the convexity of connected components of the chambers and the fact that $f_1^j(1)=0=f_2^j(1)$. In addition, for each $i=1,2$ and large $j$,
\begin{align}\notag
  \textup{$f_i^j([1-\varepsilon_1,1]\setminus \mathcal{I}_j )$ is a graph of constant curvature $\kappa_{j}$}
\end{align}
since $f_1^j<f_2^j$ implies that $(t,f_i^j(t))\in \partial G^{\delta_j}$. Since a graph of constant curvature $\kappa_j$ can be defined over an interval of length at most $2\kappa_j^{-1}$ and
$\kappa_j\to \infty$, we deduce that $\mathcal{H}^1(\mathcal{I}_j)\to \varepsilon_1$. Since $1\in \mathcal{I}_j$ for all $j$ and $G_j \cap \mathrm{int}\,\mathcal{I}_j \times [-\varepsilon_2,\varepsilon_2]=\emptyset$, we conclude that $G^{\delta_j}$ stays at positive distance from $y=\vec{e}_1$, which is a contradiction. We have thus proved \eqref{first half of hausdorff convergence}.
\par
\medskip
\noindent\textit{Step six}: In this step, we prove the other half of \eqref{hausdorff convergence of remnants in proof}, namely
\begin{align}\label{second half of hausdorff convergence}
    \sup_{x\in \Sigma} \dist (x, G^{\delta_j}) \to 0\,.
\end{align}
For such an $x$, say which is a triple junction between $S_1^0$, $S_2^0$, and $S_3^0$, by \eqref{hausdorff convergence of chambers in proof} and the definition of $\Sigma$, there exists $r_0>0$ such that given $r<r_0$, there exists $J(r)$ such that
\begin{align}\label{ball sees all three}
B_r(x) \cap S_\ell^{\delta_j}\neq \emptyset\quad\textup{ for $\ell=1,2,3$ and $j\geq J(r)$}\,.
\end{align}
Furthermore, by decreasing $r_0$ if necessary when $x\in \partial B \cap \Sigma$ is a jump point of $h$, the boundary condition \eqref{trace constraint} and absence of triple junctions for $\delta>0$ allow us to choose $1\leq \ell\leq 3$ such that
\begin{align}\label{no boundary on boundary 2}
\partial S_\ell^\delta \cap \partial B \cap  B_{r_0}(x)=\emptyset\quad\textup{for all $j$}\,.
\end{align}
Now \eqref{ball sees all three} and \eqref{no boundary on boundary 2} imply that $\partial S_\ell^{\delta_j} \cap B_r(x) \subset B$ and is also non-empty for $j\geq J(r)$. Since Theorem \ref{main regularity theorem delta positive} implies that line segments in $\partial S_\ell^{\delta_j}$ can only terminate inside $B$ at interior cusp points in $\partial G^\delta$ and $S_\ell^{\delta_j}\cap B_{r}(x)$ converges to a sector with angle strictly less than $\pi$, we find that $G^{\delta_j}\cap B_r(x)\neq \emptyset$ for all $j\geq J(r)$. Letting $r\to 0$ gives \eqref{second half of hausdorff convergence}.
\par
\medskip
\noindent\textit{Step seven}: Finally, under the assumption that $\Sigma=\{x_1,\dots, x_P\}\neq \emptyset$, we show that for large enough $j$, $G^{\delta_j}$ consists of $P$ connected components, each of which is determined by three circle arcs contained in $\partial S_{\ell_i}^{\delta_j}\cap \partial G^{\delta_j}$ for the three indices $\ell_i$, $i=1,2,3$, in the triple junction at $x$. We fix $x\in \Sigma$ which is a triple junction between the first three chambers, so there is some $B_{2r}(x)$ such that for each $\ell$, $B_{2r}(x)\cap S_\ell^0$ consists of exactly one connected component $C_\ell$ of $S_\ell^0$ for $1\leq \ell \leq 3$ (also $S_\ell^0 \cap B_{2r}(x)=\emptyset$ for $\ell\geq 4$). Up to decreasing $r$, we may also assume that
\begin{align}\label{no other sigma}
   ( \Sigma\setminus \{x\} )\cap \cl B_{2r}(x)=\emptyset\,.
\end{align}
Recalling from step two (see \eqref{hausdorff convergence of components delta to 0} and the last paragraph) that the connected components of $S_\ell^{\delta_j}$ converge in the Hausdorff sense to those of $S_\ell^0$, for $j$ large enough, we must have
\begin{align}\label{one component in the ball}
    B_r(x) \cap S_\ell^{\delta_j} = B_r(x) \cap C_\ell^j \neq \emptyset\quad 1\leq \ell \leq 3
\end{align}
for a single connected component $C_\ell^j$, and, due to \eqref{hausdorff convergence of remnants in proof} and \eqref{no other sigma},
\begin{align}\label{G only near x}
    \cl G^{\delta_j} \cap \cl B_r(x) \subset B_{r/4}(x)\,.
\end{align}
Now $\partial G_{\delta_j} \cap B_r(x)$ consists of finitely many circle arcs and has negative mean curvature (with respect to the outward normal $\nu_{G^{\delta_j}}$) along these arcs away from cusps. We claim that for $j$ large, there are precisely three such arcs, one bordering each $S_\ell^{\delta_j}$ for $1\leq \ell \leq 3$ and together bounding one connected component of $G^{\delta_j}$. There must be at least three arcs, since an open set bounded by two circle arcs has corners rather than cusps. To finish the proof, it suffices to show that there cannot be more than two distinct arcs belonging to $\partial G^{\delta_j} \cap \partial S_\ell^{\delta_j} \cap B_{r/4}(x)$ for a single $\ell\in \{1,2,3\}$. If there were, then $\partial S_\ell^{\delta_j}\cap B_r(x)$ would contain at least three distinct segments, because with only two, each of which has one endpoint outside of $B_{r}(x)$ according to \eqref{one component in the ball}-\eqref{G only near x}, one cannot resolve three cusp points as dictated by Theorem \ref{main regularity theorem delta positive}. As a consequence, there exists $\ell'\neq \ell$ such that up to a subsequence, for large $j$, there are two distinct segments, $L_1$ and $L_2$, both belonging $\partial S_\ell^{\delta_j} \cap \partial S_{\ell'}^{\delta_j} \cap B_r(x)$ and separated by at least one circle arc. It is therefore the case that $L_1$ and $L_2$ are not collinear. Also by \eqref{one component in the ball}, there is only a single convex component $C_{\ell'}^{j}$ of $S_{\ell'}^{\delta_j}$ containing $S_{\ell'}^{\delta_j}\cap B_r(x)$. Therefore, $L_1 \cup L_2\subset \partial C_{\ell}^j \cap \partial C_{\ell'}^j$. But this is impossible: since a planar convex set lies on one side of any tangent line, $\partial C_{\ell}^j$ and $\partial C_{\ell'}^j$ cannot share two non-collinear segments.\end{proof}

\begin{remark}[Explicit description of $\S^\delta$]\label{remark:explicit description remark}
    If in the conclusion of Theorem \ref{resolution for small delta corollary}, it is also the case that $\dist(\Sigma,\partial B)>f(\delta)$, so that $G^\delta \subset\!\subset B$, then $\S^\delta=(B^c,S_1'\setminus G^\delta,\dots,S_N'\setminus G^\delta,G^\delta)$ for some $\S'=(B^c,S'_1,\dots,S_N',\emptyset)$ minimizing $\mathcal{F}$ among $\mathcal{A}_0^h$. To see this, we ``excise" the wet region $G^\delta$ from $\S^\delta$. For each $x\in \Sigma$, divide $G^\delta\cap B_{f(\delta)}(x)$ into three pieces, each bounded by an arc $A_x^\ell\subset\partial G^\delta \cap  \partial S_\ell^\delta$ for some $\ell$ and the two segments $B_x^\ell,C_x^\ell$ connecting the endpoints of $A_x^\ell$ to the centroid of $G^\delta \cap B_{f(\delta)}(x)$. Since all $A_x^\ell$ have equal lengths/curvatures (by $G^\delta \subset\!\subset B$ and the fact that there is only one configuration up to isometries of three mutually tangent circles with radius $r$), all the pieces of $G^\delta$ are congruent, with area $A_\delta = \delta/(3\mathcal{H}^0(\Sigma))$ such that $\mathcal{H}^1(A_x^\ell)-\mathcal{H}^1(B_x^\ell\cup C_x^\ell)=c\sqrt{A_\delta}$ for a constant $c<0$. For each $\ell$, define $G_\ell^\delta$ as the union of all pieces of $G^\delta$ with a boundary arc in $\partial S_\ell^\delta \cap \partial G^\delta$. Let $\S'=(B^c,S_1^\delta \cup G_1^\delta, \dots, S_N^\delta \cup G_N^\delta,\emptyset)$. By the definition of $\S'$ and minimality of $\S^0$,
\begin{align}\notag
   \mathcal{F}(\S^\delta;B)-3\mathcal{H}^0(\Sigma)c\sqrt{A_\delta}   =  \mathcal{F}(\S';B)\geq \mathcal{F}(\S^0;B).
\end{align}
Since this lower bound for the minimum of $\mathcal{F}$ on $\mathcal{A}_\delta^h$ is achieved by the construction wetting the singularities of $\S^0$ as in Figure \ref{perturbation figure}, it follows that $\mathcal{F}(\S';B)=\mathcal{F}(\S^0;B)$ and $\S'$ is minimizing for $\mathcal{F}$.
\end{remark}

\noindent{\bf Data Availability:} Data sharing not applicable to this article as no datasets were generated or analyzed during the current study.

\bibliographystyle{alpha}
\bibliography{references}
\end{document}